\documentclass[11pt]{article}
\usepackage[width=15.5cm,height=21cm]{geometry}
\usepackage[colorlinks,allcolors=blue]{hyperref}

\usepackage{cite}
\usepackage{url}
\usepackage{comment}
\newcommand{\doi}[1]{\href{https://doi.org/#1}{\url{doi:#1}}}
\newcommand{\myurl}[1]{\href{#1}{\url{#1}}}

\usepackage[intlimits]{amsmath}
\usepackage{amsfonts,amssymb,amsthm}

\newtheorem{thm}{Theorem}[section]
\newtheorem{cor}[thm]{Corollary}
\newtheorem{lem}[thm]{Lemma}
\newtheorem{prop}[thm]{Proposition}
\newtheorem{problem}[thm]{Problem}
\theoremstyle{definition}
\newtheorem{defn}[thm]{Definition}
\newtheorem{example}[thm]{Example}
\newtheorem{rem}[thm]{Remark}
\numberwithin{equation}{section}

\newcommand{\bC}{\mathbb{C}}
\newcommand{\bD}{\mathbb{D}}
\newcommand{\bR}{\mathbb{R}}
\newcommand{\bN}{\mathbb{N}}
\newcommand{\bH}{\mathbb{H}}

\newcommand{\cH}{\mathcal{H}}

\newcommand{\cP}{\mathcal{P}}
\newcommand{\cQ}{\mathcal{Q}}

\newcommand{\al}{\alpha}

\newcommand{\Ga}{\Gamma}
\newcommand{\ga}{\gamma}
\newcommand{\de}{\delta}
\newcommand{\eps}{\varepsilon}
\newcommand{\la}{\lambda}
\renewcommand{\phi}{\varphi}

\newcommand{\enumber}{\operatorname{e}}
\newcommand{\iu}{\operatorname{i}}
\renewcommand{\Re}{\operatorname{Re}}
\renewcommand{\Im}{\operatorname{Im}}
\newcommand{\dif}{\mathrm{d}}

\newcommand{\AC}{\operatorname{AC}}
\newcommand{\sinc}{\operatorname{sinc}}

\newcommand{\dsqrt}{d_{\operatorname{sqrt}}}
\newcommand{\conj}[1]{\overline{#1}}
\newcommand{\topequiv}{\stackrel{\operatorname{top}}{\sim}}
\newcommand{\lipequiv}{\stackrel{\operatorname{Lip}}{\sim}}
\newcommand{\uniequiv}{\stackrel{\operatorname{u}}{\sim}}
\newcommand{\intrequiv}{\stackrel{\operatorname{intr}}{\sim}}

\newcommand{\infsim}{\stackrel{\operatorname{infin}}{\sim}}
\newcommand{\strsim}{\stackrel{\operatorname{str}}{\sim}}
\newcommand{\dil}{\operatorname{dil}}
\newcommand{\id}{\operatorname{id}}

\usepackage{mathtools}
\newcommand{\eqdef}{\coloneqq}

\usepackage{tikz}
\usetikzlibrary{arrows, arrows.meta, calc}

\definecolor{darkgreen}{rgb}{0, 0.5, 0}
\definecolor{lightgreen}{rgb}{0.7, 1, 0.7}
\definecolor{lightblue}{rgb}{0.7, 0.7, 1}

\tikzset{mynode/.style
={rounded corners,
text centered,
text width=10ex,
minimum size = 4ex,
draw = darkgreen,
fill = lightgreen,
outer sep = 0.4ex}}

\tikzset{myedge/.style
={black, very thick, -Stealth}}

\tikzset{badedge/.style
={red, very thick, -Stealth}}

\newcommand*{\StrikeThruDistance}{0.12cm}%

\newcommand*{\StrikeOne}{\StrikeThruDistance,\StrikeThruDistance}%

\newcommand*{\StrikeTwo}{\StrikeThruDistance,
-\StrikeThruDistance}

\title{Locally similar distances\\
and equality of the induced intrinsic distances}
\author{Erick Lee-Guzm\'{a}n,
Egor A. Maximenko,\\
Enrique Abdeel Mu\~{n}oz-de-la-Colina,
Marco Iv\'{a}n Ruiz-Carmona}

\begin{document}
\maketitle

\begin{abstract}
Let $X$ be a set and $d_1,d_2$ be two distances on $X$.
We say that $d_1$ and $d_2$ are locally similar and write $d_1\cong d_2$
if $d_1$ and $d_2$ are topologically equivalent and, for every $a$ in $X$,
\[
\lim_{x\to a}
\frac{d_2(x,a)}{d_1(x,a)}=1.
\]
We prove that if $d_1\cong d_2$,
then the intrinsic distances induced by $d_1$ and $d_2$ coincide.
We also provide sufficient conditions for $d_1\cong d_2$ and consider several examples related to reproducing kernel Hilbert spaces.


\end{abstract}

\medskip
\tableofcontents

\medskip
\section{Introduction}

\subsection*{Background}

The concept of distance is simple, but it plays a fundamental role in mathematics. 
See \cite{gromov_riemannian, burago_metric, kunzinger_alexandrov,
deza_encyclopedia}
for many examples and applications of distances.
There are many different notions of ``equivalent distances''
on a fixed set $X$,
including the following standard notions:
\begin{itemize}
\item $d_1$ and $d_2$ are \emph{topologically equivalent}
($d_1\topequiv d_2$),
if they induce the same topology;
\item $d_1$ and $d_2$ are
\emph{uniformly equivalent}
($d_1\uniequiv d_2$),
if they induce the same uniform structure,
i.e., the identity function $(X,d_1)\to(X,d_2)$
and its inverse are uniformly continuous;
\item $d_1$ and $d_2$ are
\emph{biLipschitz equivalent}
($d_1\lipequiv d_2$),
if the identity function $(X,d_1)\to(X,d_2)$ and its inverse are Lipschitz continuous.
\end{itemize}

Every distance $d$ induces the
\emph{intrinsic distance} $d^\ast$,
see~\cite[Chapter 2]{burago_metric} or~\cite[Chapter 1]{kunzinger_alexandrov}. 
Obviously, the binary relation $d_1\intrequiv d_2$
(``intrinsic equivalence''),
which we define by $d_1^\ast = d_2^\ast$, is an equivalence relation on the set of all distances defined on a fixed domain.
A natural problem is to obtain criteria for
$d_1^\ast = d_2^\ast$.
It is easy to see that $d_1\lipequiv d_2$ implies
$d_1^\ast\lipequiv d_2^\ast$
(Proposition~\ref{prop:lipequiv_implies_lipequiv_for_intrinsic_distances}),
but it does not imply
$d_1^\ast=d_2^\ast$
(Example~\ref{example:pseudologarithmic_distance_on_segment}).

\subsection*{Main results}

In this paper, we introduce the concept of \emph{locally similar} distances
(Definition~\ref{def:locally_similar_distances})
and denote it by $d_1\cong d_2$.
This concept is related to some concepts
introduced by Capogna
and Le~Donne~\cite{Capogna_LeDonne_2021};
see Remark~\ref{rem:connection_to_Capogna_LeDonne}.

Our main result (Theorem~\ref{thm:main})
states that the relation $d_1\cong d_2$ is sufficient for $d_1^\ast=d_2^\ast$.
We show that the relation $\cong$ is useful to compute $d^\ast$ in various examples.
Nevertheless,
$d_1\cong d_2$ is not necessary for $d_1^\ast=d_2^\ast$;
see Examples~\ref{example:sqrt}
and~\ref{example:comb}.

By definition, locally similar distances are topologically equivalent.
We show that $\cong$ is independent of $\lipequiv$, and $\uniequiv$:
the biLipschitz equivalence does not imply the local similarity
(Example~\ref{example:biLipschitz_equiv_does_not_imply_loc_sim}),
and the local similarity does not imply even the uniform equivalence
(Example~\ref{example:open_interval_with_circular_distance}).
Figure~\ref{fig:logic_relations} shows
the logical relationships among
the equivalencies of distances considered in the paper.

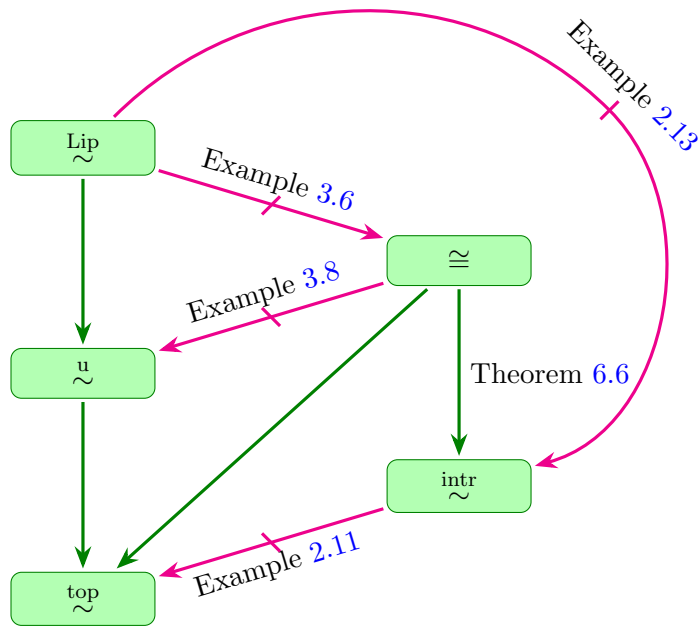
\begin{figure}[htb]
\centering
\begin{tikzpicture}
\node [mynode] (strsim)
  at (0, 5)
  {$\strsim$};
\node [mynode] (infinsim)
  at (0, 2.5)
  {$\infsim$};
\node [mynode] (loc)
  at (0, 0)
  {$\cong$};
\node [mynode] (intr)
  at (0, -2.5)
  {$\intrequiv$};
\node [mynode] (lip)
  at (-5, 2.5)
  {$\lipequiv$};
\node [mynode] (uni)
  at (-5, 0)
  {$\uniequiv$};
\node [mynode] (top)
  at (-5, -2.5)
  {$\topequiv$};
\draw [myedge]
  (loc) -- (intr)
  node [black, midway, left]
  {Thm.~\ref{thm:main}};
\draw [badedge]
  (intr) to [out=45, in=-45] (loc);
\coordinate (MidWaySeven)
  at (0.77, -1.25);
\draw [very thick, magenta, -]
  ($(MidWaySeven)-(\StrikeTwo)$)
  -- 
  ($(MidWaySeven)+(\StrikeTwo)$);
\node at (1, -1.25) [right]
  {Ex.~\ref{example:comb}};
\draw [myedge]
  (infinsim) -- (loc)
  node [black, midway, left]
  {Pr.~\ref{prop:infsim_implies_locsim}};
\draw [badedge]
  (loc) to [out=45, in=-45] (infinsim);
\coordinate (MidWayFive)
  at (0.77, 1.25);
\draw [very thick, magenta, -]
  ($(MidWayFive)-(\StrikeTwo)$)
  -- 
  ($(MidWayFive)+(\StrikeTwo)$);
\node at (1, 1.25) [right]
  {Ex.~\ref{example:hook}};
\draw [myedge]
  (strsim) -- (infinsim)
  node [black, midway, left]
  {Pr.~\ref{prop:strsim_implies_infsim}};
\draw [badedge]
  (infinsim) to [out=45, in=-45] (strsim);
\coordinate (MidWaySix)
  at (0.77, 3.75);
\draw [very thick, magenta, -]
  ($(MidWaySix)-(\StrikeTwo)$)
  -- 
  ($(MidWaySix)+(\StrikeTwo)$);
\node at (1, 3.75) [right]
  {Ex.~\ref{example:OmerCantor}};
\draw [badedge]
  (strsim.south west) to (uni);
\coordinate (MidWayEight)
  at (-2.65, 2.7);
\draw [very thick, magenta, -]
  ($(MidWayEight)-(\StrikeTwo)$)
  -- 
  ($(MidWayEight)+(\StrikeTwo)$);
\node at ($(MidWayEight)+(-0.25cm,0.25cm)$)
  [rotate=48]
  {Rem.~\ref{rem:strlocsim_not_unisim}};
\draw [myedge]
  (lip) -- (uni);
\draw [myedge]
  (uni) -- (top);
\draw [myedge]
  (loc) -- (top)
  node [black, midway, above, sloped]
  {(by def.)};
\draw [badedge]
  (lip) -- (loc)
  node [black, midway, above, sloped] {Ex.~\ref{example:biLipschitz_equiv_does_not_imply_loc_sim}};
\draw [badedge]
  (loc) -- (uni)
  node [black, midway, above, sloped] {Ex.~\ref{example:open_interval_with_circular_distance}};
\coordinate (MidWayOne)
  at ($(lip.south east)!0.5!(loc.north west)$);
    \draw [very thick, magenta, -]
    ($(MidWayOne)-(\StrikeOne)$) -- 
    ($(MidWayOne)+(\StrikeOne)$);
\coordinate (MidWayTwo)
  at ($(uni.north east)!0.5!(loc.south west)$);
\draw [very thick, magenta, -]
  ($(MidWayTwo)-(\StrikeTwo)$)
  -- 
  ($(MidWayTwo)+(\StrikeTwo)$);
\draw [badedge]
  (intr) -- (top)
  node [black, midway, below, sloped] {Ex.~\ref{example:quadrant_angular}};
\coordinate (MidWayThree)
  at ($(top.north east)!0.5!(intr.south west)$);
\draw [very thick, magenta, -]
  ($(MidWayThree)-(\StrikeTwo)$)
  -- 
  ($(MidWayThree)+(\StrikeTwo)$);
\draw [badedge]
  (lip)
  to [out=-135, in=135]
  (-6.5, -3.5)
  to[out=-45, in=-135]
  (intr);
\coordinate (MidWayFour)
  at (-6.5, -3.5);
\draw [very thick, magenta, -]
  ($(MidWayFour)-(\StrikeOne)$)
  -- 
  ($(MidWayFour)+(\StrikeOne)$);
\node at (-6.8, -3.8)
 [rotate=-45]
{Ex.~\ref{example:pseudologarithmic_distance_on_segment}};
\end{tikzpicture}
\caption{Some logic relations between different equivalencies of distances.
We give references to the corresponding theorems, propositions, remarks,
and examples (counter-examples).
\label{fig:logic_relations}
}
\end{figure}

Several authors have studied the situation when $d$ is an intrinsic distance,
i.e., $d^\ast=d$
(see 
\cite[Section~1.3]{kunzinger_alexandrov},
\cite[Chapter~1, parts A, B]{gromov_riemannian},
\cite[Section~2.3]{burago_metric}).
The constructions of this paper are useful in a more general case,
namely when $d^\ast\cong d$.

We hope to relate this work to Riemannian metrics in future research.
In this paper, we avoid Riemannian metrics
to show the power of more elementary tools.

\subsection*{Structure of the paper}

In Section~\ref{sec:locally_similar_distances},
we introduce the concept of locally similar distances and show that it is not comparable with $\lipequiv$ or $\uniequiv$.
In Section~\ref{sec:strongly_similar_distances},
we study two other concepts
(\emph{strong local similarity} $\strsim$
and \emph{infinitesimal similarity} $\infsim$)
that are slightly stronger
than local similarity.
In Section~\ref{sec:locally_similar_distances_and_composition}, we show that composing a distance $d_1$ with a function that satisfies certain properties yields a distance $d_2$ that is locally similar to $d_1$
($d_2\cong d_1$, if $d_2=f\circ d_1$).
In Section~\ref{sec:intrinsic_distance}, we recall the definition of the intrinsic distance associated with a given distance;
we also provide examples which are used throughout the paper.
Section~\ref{sec:main_result} contains a proof of the main result.
Section~\ref{sec:examples_RK} contains several examples related to the distances induced by kernels (or, equivalently, by reproducing kernel Hilbert spaces).

\medskip 
\section{Locally similar distances}
\label{sec:locally_similar_distances}

In this section, we introduce the main concept of this paper (Definition~\ref{def:locally_similar_distances})
and study its elementary properties.
We have invented this concept
as a sufficient condition
for the intrinsic equivalence,
see Theorem~\ref{thm:main}.

\begin{defn}
\label{def:locally_similar_distances}
Let $X$ be a set and
$d_1,d_2\colon X\to[0,+\infty)$
be distances on $X$.
We say that $d_1$ and $d_2$ are
\emph{locally similar}
and write $d_1\cong d_2$
if $d_1$ and $d_2$ are topologically equivalent and for every $a$ in $X$,
\begin{equation}
\label{eq:def_loc_sim}
\lim_{x\to a}
\frac{d_2(x,a)}{d_1(x,a)}=1.
\end{equation}
\end{defn}

In the context of Definition~\ref{def:locally_similar_distances},
let $\tau$ be the topology induced by $d_1$ (or $d_2$), and let $\tau(a)$ be the set of all open neighborhoods of $a$.
Then,~\eqref{eq:def_loc_sim} can be written in the following more explicit form:
\begin{equation}
\label{eq:def_loc_sim_explicit}
\forall \eps>0\qquad
\exists U\in\tau(a)\qquad
\forall x\in U\setminus\{a\}\qquad
\left|\frac{d_2(x,a)}{d_1(x,a)}-1\right|<\eps.
\end{equation}

\begin{prop}
Let $X$ be a set.
The binary relation $\cong$,
defined on the set of all distances on $X$,
is an equivalence relation.
\end{prop}

\begin{proof}
It follows easily
from the arithmetic properties of limits.
\end{proof}

Given a metric space $(X,d)$,
for every $a$ in $X$
and every $\de>0$
we denote by $B_d(a,\de)$ the open ball
\[
B_d(a,\de)\eqdef
\bigl\{x\in X\colon\ d(x,a)<\de\bigr\}.
\]
Furthermore, we denote by $\tau_d$ the topology induced by $d$.

\hyphenation{to-po-lo-gy}
In the next proposition,
we characterize the relation $d_1\cong d_2$ 
avoiding the induced topology.

\begin{prop}
\label{prop:local_similarity_another_form}
Let $X$ be a set and
$d_1,d_2\colon X\to[0,+\infty)$
be distances on $X$.
$d_1\cong d_2$ if and only if for every $a$ in $X$, the following limit relations hold:
\begin{align}
\label{eq:loc_sim_first_limit}
\lim_{\substack{d_1(x,a)\to 0\\x\ne a}}
\frac{d_2(x,a)}{d_1(x,a)}=1,
\\
\label{eq:loc_sim_second_limit}
\lim_{\substack{d_2(x,a)\to 0\\x\ne a}}
\frac{d_1(x,a)}{d_2(x,a)}=1.
\end{align}
\end{prop}

We understand~\eqref{eq:loc_sim_first_limit} and~\eqref{eq:loc_sim_second_limit}
in the following sense, respectively:
\begin{align}
\label{eq:loc_sim_first_limit_explicit}
\forall\eps>0\qquad
\exists\de>0\qquad
\forall x\in B_{d_1}(a,\de)\setminus\{a\}
\qquad
\left|\frac{d_2(x,a)}{d_1(x,a)}-1\right|<\eps,
\\[1ex]
\label{eq:loc_sim_second_limit_explicit}
\forall\eps>0\qquad
\exists\de>0\qquad
\forall x\in B_{d_2}(a,\de)\setminus\{a\}
\qquad
\left|\frac{d_1(x,a)}{d_2(x,a)}-1\right|<\eps.
\end{align}

\begin{proof}
Necessity.
Suppose that $d_1\cong d_2$, i.e.,
\eqref{eq:def_loc_sim} holds for each $a$.
Since $\tau_{d_1}$ is induced by $d_1$,
there exists $\de>0$ such that
$B_{d_1}(a,\de)\subseteq U$.
Thereby, we get~\eqref{eq:loc_sim_first_limit_explicit}.
By the arithmetic properties of limits,
\eqref{eq:def_loc_sim}
implies a similar limit with the quotient $d_1/d_2$,
which yields~\eqref{eq:loc_sim_second_limit_explicit}.

Sufficiency.
Suppose
that~\eqref{eq:loc_sim_first_limit_explicit}
and~\eqref{eq:loc_sim_second_limit_explicit} are fulfilled for every $a$ in $X$.
Let $a\in X$.
Applying~\eqref{eq:loc_sim_first_limit_explicit} with $\eps=1/2$ we get $\de>0$ such that for all $x$ in $B_{d_1}(a, \de)$,
\[
d_2(x,a)\le \frac{3}{2}d_1(x,a),
\]
which implies that
$B_{d_1}(a,r)\subseteq B_{d_2}(a,3r/2)$ for every $r$ with $0<r<\de$.
Since $a$ is arbitrary,
we get $\tau_{d_2}\subseteq\tau_{d_1}$.
In a similar way, using~\eqref{eq:loc_sim_second_limit},
we obtain that $\tau_{d_1}\subseteq\tau_{d_2}$.
So, $d_1$ and $d_2$ induce the same topology.
Finally, taking $U\eqdef B_{d_1}(a,\de)$,
we see that~\eqref{eq:loc_sim_first_limit_explicit} implies~\eqref{eq:def_loc_sim_explicit}.
\end{proof}

\begin{rem}
\label{rem:generalization_to_pseudodistances}
This paper could be easily generalized to pseudodistances.
Instead of $x\ne a$,
we should require $d_1(x,a)>0$
in~\eqref{eq:def_loc_sim},
\eqref{eq:def_loc_sim_explicit},
\eqref{eq:loc_sim_first_limit},
and~\eqref{eq:loc_sim_first_limit_explicit},
and $d_2(x,a)>0$ in~\eqref{eq:loc_sim_second_limit}
and~\eqref{eq:loc_sim_second_limit_explicit}.
Of course, the topological equivalence of $d_1$ and $d_2$ implies that $d_1(x,a)>0$ if and only if $d_2(x,a)>0$.
\end{rem}

\begin{rem}
\label{rem:connection_to_Capogna_LeDonne}
The binary relation $\cong$
can be described in terms of some concepts
from
Capogna and Le~Donne~\cite[Subsection~2.1]{Capogna_LeDonne_2021}.
Given two metric spaces $(X,d_1)$ and $(Y,d_2)$,
a homeomorphism $f\colon X\to Y$,
and a point $a$ in $X$,
they consider the following superior and inferior limits:
\begin{equation}
\label{eq:Lf_and_lf}
L_f(a)\eqdef
\limsup_{x\to a}
\frac{d_2(f(x),f(a))}{d_1(x,a)},
\qquad
\ell_f(a)\eqdef
\liminf_{x\to a}
\frac{d_2(f(x),f(a))}{d_1(x,a)}.
\end{equation}
We apply these concepts in the particular case where $Y=X$ and $f$ is the identity function
$\id_X\colon X\to X$
where the domain is equipped with $d_1$
and the codomain is equipped with $d_2$.
It is well known that the existence of $\lim$ is equivalent to the equality between
$\limsup$ and $\liminf$;
see, e.g.,
\cite[Chapter~IV, Subsection~5.6]{Bourbaki1995top}.
Thus,
$d_1\cong d_2$ if and only if
for every $a$ in $X$,
$L_{\id_X}(a)=\ell_{\id_X}(a)=1$.
In the terminology from~\cite{Capogna_LeDonne_2021},
this means that $\id_X$ has
pointwise Lipschitz constant $1$
and distortion $1$ at every point.
\end{rem}

\begin{example}
\label{example:biLipschitz_equiv_does_not_imply_loc_sim}
There exist two distances that are biLipschitz equivalent,
but not locally similar.
Consider $X = [0, 1]$. Define $\rho\colon X^2\to [0,+\infty)$, $\rho(x, y) \eqdef |f(x)-f(y)|$, where
\[
f(x) \eqdef x(1+x).
\]
More explicitly, 
\[
\rho(x, y) = |x-y|(1+x+y).
\]
Let $d_{[0,1]}$ be the Euclidean distance on $[0,1]$.
Obviously, $\rho$ is biLipschitz equivalent to $d_{[0,1]}$:
\[
|x-y|
\leq 
\rho(x, y)
\leq 
3|x-y|.
\]
We observe that the limit from Definition~\ref{def:locally_similar_distances} exists for each $a$ in $X$, but it depends on $a$:
\[
\lim_{x\to a}
\frac{\rho(x,a)}{d_{[0,1]}(x,a)}
= 
\lim_{x\to a} (1+x+a)
=
1+2a.
\]
Therefore, $\rho$ is not locally similar to $d_{[0,1]}$.
Moreover, there is no $C>0$ such that
$\rho \cong C d_{[0,1]}$.
\hfill\qedsymbol
\end{example}

\begin{example}
\label{example:loc_sim_does_not_imply_biLipschitz_equiv}
There exist two distances that are locally similar,
but not biLipschitz equivalent.  
Consider $(\bR, d_\bR)$ and $(\bR, \rho)$, where 
    \[
    \rho(x, y) \eqdef \min\{1, |x-y|\}.
    \]
    Then, $d_\bR\cong \rho$ but there is no $C>0$ such that $d_\bR(x, y)\leq C \rho(x, y)$ for every $x,y$ in $\bR$. 
    \hfill\qedsymbol
\end{example}

\begin{example}
\label{example:open_interval_with_circular_distance}
There exist two distances that are locally similar, but not uniformly equivalent.
We consider the open
interval $X=(0,2\pi)$
provided with the Euclidean distance
$d_X(a,b) \eqdef |a-b|$
and the ``circular distance''
\[
\rho(a,b) \eqdef
\bigl|\enumber^{\iu a}-\enumber^{\iu b}\bigr|,
\qquad\text{i.e.},\qquad
\rho(a,b)=2\left|\sin\frac{a-b}{2}\right|,
\]
see Figure~\ref{fig:circular_distance}.
It is easy to see that $d_X\cong \rho$.
Nevertheless, for $x_m=\frac{1}{m}$ and $y_m=2\pi-\frac{1}{m}$,
\[
\lim_{m\to\infty}\rho(x_m,y_m)=0,\qquad
\lim_{m\to\infty}d_X(x_m,y_m)=2\pi.
\]
Therefore, the identity function $(X,\rho)\to (X,d_X)$ is not uniformly continuous.
Hence, $d_X$ and $\rho$ are not uniformly equivalent.
\hfill\qedsymbol
\end{example}

\begin{figure}[htb]
\centering
\begin{tikzpicture}
\filldraw [lightgreen]
  (0.5, -0.1) rectangle (5.2832, 0.1);
\draw (0, 0) -- (2*3.1416, 0);
\filldraw (0.5, 0) circle (0.05cm);
\node at (0.5, 0) [below] {$a$};
\filldraw (5.2832, 0)
  circle (0.05cm);
\node at (5.2832, 0) [below] {$b$};
\filldraw [black, fill=white]
  (0, 0) circle (0.05cm);
\filldraw [black, fill=white]
  (2*3.1416, 0) circle (0.05cm);
\node at (0, 0) [above] {$0$};
\node at (2*3.1416, 0) [above] {$2\pi$};
\begin{scope}[xshift=10cm]
\draw (0, 0) circle (1cm);
\filldraw [draw=black, fill=white]
  (1, 0) circle (0.05cm);
\draw [thick, blue]
  (0.8776, 0.4794)
  -- (0.5403, -0.8415);
\filldraw (0.8776, 0.4794)
  circle (0.05cm);
\node at (0.8776, 0.4794) [above right]
  {$\enumber^{\iu a}$};
\filldraw (0.5403, -0.8415)
  circle (0.05cm);
\node at (0.5403, -0.8415) [below right]
  {$\enumber^{\iu b}$};
\end{scope}
\end{tikzpicture}
\caption{Distances from Example~\ref{example:open_interval_with_circular_distance}.
\label{fig:circular_distance}}
\end{figure}
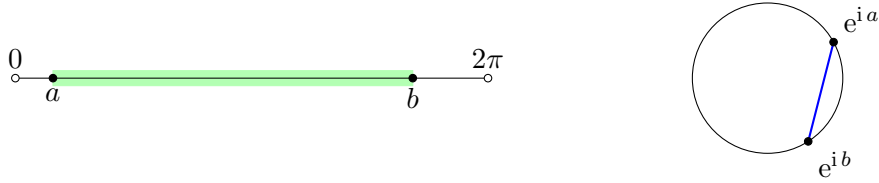

\section{Strongly locally similar
and infinitesimally similar distances}
\label{sec:strongly_similar_distances}

In this section,
we give two alternative concepts
for local similarity of distances
(Definitions~\ref{def:strongly_similar_distances} and~\ref{def:infin_similar})
and show that they are slightly stronger than $\cong$ from Definition~\ref{def:locally_similar_distances}.
The concepts introduced in this section
are close to definitions from Gromov~\cite[Definition~1.1]{gromov_riemannian} and Lytchak~\cite[Definition~3.2]{Lytchak2005}.

\subsection*{Strongly locally similar distances}

\begin{defn}
\label{def:strongly_similar_distances}
Let $X$ be a set and
$d_1,d_2\colon X\to[0,+\infty)$
be distances on $X$.
We say that $d_1$ and $d_2$ are
\emph{strongly locally similar}
and write $d_1\strsim d_2$
if $d_1$ and $d_2$ are topologically equivalent and for every $a$ in $X$,
\begin{equation}
\label{eq:def_strongly_similar}
\lim_{\substack{(x,y)\to(a,a)\\x\ne y}}
\frac{d_2(x,y)}{d_1(x,y)}=1.
\end{equation}
\end{defn}

In the context of Definition~\ref{def:strongly_similar_distances},
let $\tau$ be the topology induced by $d_1$ (or $d_2)$, and let $\tau(a)$ be the set of all open neighborhoods of $a$.
Furthermore, let $\Delta(X)$ be the
``diagonal'' of $X^2$:
\[
\Delta(X) \eqdef \{(x,y)\in X^2\colon\ x=y\}.
\]
The limit in~\eqref{eq:def_strongly_similar} is understood in the sense of the product topology on $X^2$.
More explicitly,~\eqref{eq:def_strongly_similar} means that
\begin{equation}
\label{eq:str_sim_explicit}
\forall \eps>0\qquad
\exists U\in\tau(a)\qquad
\forall (x,y)\in U^2\setminus\Delta(X)\qquad
\left|\frac{d_2(x,y)}{d_1(x,y)}-1\right|<\eps.
\end{equation}

\begin{prop}
Let $X$ be a set.
The binary relation $\strsim$,
defined on the set of all distances on $X$,
is an equivalence relation.
\end{prop}

\begin{proof}
It follows easily from the arithmetic properties of limits.
\end{proof}

In the next proposition,
we characterize the condition $d_1\strsim d_2$ 
avoiding the induced topology.

\begin{prop}
\label{prop:strongly_similar_another_form}
Let $X$ be a set and
$d_1,d_2\colon X\to[0,+\infty)$ be distances on $X$.
Then, $d_1\strsim d_2$ if and only if for every $a$ in $X$, the following limit relations hold:
\begin{align}
\label{eq:str_sim_first_limit}
\lim_{\substack{d_1(x,a)\to 0\\d_1(y,a)\to 0\\x\ne y}}
\frac{d_2(x,y)}{d_1(x,y)}=1,
\\
\label{eq:str_sim_second_limit}
\lim_{\substack{d_2(x,a)\to 0\\d_2(y,a)\to 0\\x\ne y}}
\frac{d_1(x,y)}{d_2(x,y)}=1.
\end{align}
\end{prop}

We understand~\eqref{eq:str_sim_first_limit} and~\eqref{eq:str_sim_second_limit} in the following sense, respectively:
\begin{align}
\label{eq:str_sim_first_limit_explicit}
\forall\eps>0\qquad
\exists\de>0\qquad
\forall (x,y)\in B_{d_1}(a,\de)^2\setminus\Delta(X)
\qquad
\left|\frac{d_2(x,y)}{d_1(x,y)}-1\right|<\eps,
\\[1ex]
\label{eq:str_sim_second_limit_explicit}
\forall\eps>0\qquad
\exists\de>0\qquad
\forall (x,y)\in B_{d_2}(a,\de)^2\setminus\Delta(X)
\qquad
\left|\frac{d_1(x,y)}{d_2(x,y)}-1\right|<\eps.
\end{align}

\begin{proof}
Necessity.
Suppose that $d_1\strsim d_2$, i.e.,
\eqref{eq:str_sim_explicit}
holds for every $a$ in $X$.
Since $\tau_{d_1}$ is induced by $d_1$, there exists $\de>0$ such that $B_{d_1}(a,\de)\subseteq U$.
Thereby we get~\eqref{eq:str_sim_first_limit_explicit}.
By the arithmetic properties of limits,
\eqref{eq:def_strongly_similar}
implies a similar limit with the quotient $d_1/d_2$,
which yields~\eqref{eq:str_sim_second_limit_explicit}.

Sufficiency.
Suppose that~\eqref{eq:str_sim_first_limit_explicit}
and~\eqref{eq:str_sim_second_limit_explicit} are fulfilled for every $a$ in $X$.
Applying~\eqref{eq:str_sim_first_limit_explicit} with $\eps=1/2$ and $y=a$ we get $\de>0$ such that for all $x$ in $B_{d_1}(a, \de)$,
\[
d_2(x,a)\le \frac{3}{2}d_1(x,a),
\]
which implies that
$B_{d_1}(a,r)\subseteq B_{d_2}(a,3r/2)$ for every $r$ with $0<r<\de$.
Since $a$ is arbitrary, we get
$\tau_{d_2}\subseteq\tau_{d_1}$.
In a similar way, using~\eqref{eq:str_sim_second_limit_explicit},we obtain that
$\tau_{d_1}\subseteq\tau_{d_2}$.
So, $d_1$ and $d_2$ induce the same topology.
Finally, taking $U\eqdef B_{d_1}(a,\de)$
we see that~\eqref{eq:str_sim_first_limit_explicit} implies~\eqref{eq:str_sim_explicit}.
\end{proof}

\begin{rem}
\label{rem:strlocsim_not_unisim}
In Example~\ref{example:open_interval_with_circular_distance},
the distances $d_X$ and $\rho$ are strongly locally similar but not uniformly equivalent.
\end{rem}

\subsection*{Description of strong local similarity in terms of local dilatation}

Suppose that
$(X,d_X)$ and $(Y,d_Y)$ are metric spaces,
$f\colon X\to Y$, and $a\in X$.
According to Gromov~\cite[Definition~1.1]{gromov_riemannian}
the \emph{local dilatation} of $f$ at the point $a$ is defined as
\begin{equation}
\label{eq:local_dilatation_def}
\dil_a(f)
\eqdef
\lim_{\de\to 0}\;
\sup_{\substack{x,y\in B_{d_1}(a,\de)\\x\ne y}}
\frac{d_2(f(x),f(y))}{d_1(x,y)}.
\end{equation}
This expression can be written as the following upper limit:
\begin{equation}
\label{eq:local_dilatation_as_limsup}
\dil_a(f)
=
\limsup_{\substack{d_1(x,a)\to 0\\
d_1(y,a)\to 0\\
(x,y)\notin\Delta(X)}}
\frac{d_2(f(x),f(y))}{d_1(x,y)}.
\end{equation}

\begin{prop}
Let $X$ be a set and $d_1,d_2$ be distances on $X$.
Then, the following conditions are equivalent:
\begin{itemize}
\item[(a)] $d_1\strsim d_2$;
\item[(b)] the identity function $(X,d_1)\to (X,d_2)$ and its inverse
$(X,d_2)\to(X,d_1)$
have local dilatation $1$ at every point of $X$.
\end{itemize}
\end{prop}

\begin{proof}
(b)$\Rightarrow$(a).
Suppose (b).
This means that for every $a$ in $X$,
\begin{equation}
\label{eq:loc_dil_1}
\lim_{\de\to 0}\
\sup_{\substack{x,y\in B_{d_1}(a,\de)\\x\ne y}}
\frac{d_2(x,y)}{d_1(x,y)}
= 1
\end{equation}
and
\begin{equation}
\label{eq:loc_dil_2}
\lim_{\de\to 0}\
\sup_{\substack{x,y\in B_{d_2}(a,\de)\\x\ne y}}
\frac{d_1(x,y)}{d_2(x,y)}
= 1.
\end{equation}
For every $a$ in $X$,
the limit relations~\eqref{eq:loc_dil_1} 
and~\eqref{eq:loc_dil_2} 
imply that for every $\eps>0$
there exists $\de>0$ such that
\begin{equation}
\label{eq:loc_dil_1_less}
\forall (x,y)\in B_{d_1}(a,\de)^2\setminus\Delta(X)
\qquad
\frac{d_2(x,y)}{d_1(x,y)}<1+\eps
\end{equation}
and
\begin{equation}
\label{eq:loc_dil_2_less}
\forall (x,y)\in B_{d_2}(a,\de)^2\setminus\Delta(X)
\qquad
\frac{d_1(x,y)}{d_2(x,y)}<1+\eps.
\end{equation}
Similarly to the proof of Proposition~\ref{prop:strongly_similar_another_form}, we can deduce that 
$\tau_{d_1}=\tau_{d_2}$.
After that, for each $a$ in $X$,
from~\eqref{eq:loc_dil_1_less} and~\eqref{eq:loc_dil_2_less}
we get $U\eqdef B_{d_1}(a,\de)\cap B_{d_2}(a,\de)$ in $\tau(a)$ such that
\[
\forall (x,y)\in U^2\setminus\Delta(X)
\qquad
\frac{1}{1+\eps}
<\frac{d_2(x,y)}{d_1(x,y)}<1+\eps.
\]
Therefore, \eqref{eq:str_sim_explicit} holds.

\medskip\noindent
(a)$\Rightarrow$(b).
The relations~\eqref{eq:str_sim_first_limit}
and~\eqref{eq:str_sim_second_limit}
easily imply
the formulas~\eqref{eq:loc_dil_1}
and~\eqref{eq:loc_dil_2}
with superior limits.
\end{proof}

\medskip
\subsection*{Infinitesimally similar distances}

The following concept is inspired by
Lytchak~\cite[Definition~3.2]{Lytchak2005}.

\begin{defn}
\label{def:infin_similar}
Let $X$ be a set and $d_1,d_2\colon X\to[0,+\infty)$ be distances on $X$.
We say that $d_1$ and $d_2$ are \emph{infinitesimally similar} and write
$d_1\infsim d_2$
if $d_1$ and $d_2$ are topologically equivalent and
for every $a$ in $X$,
\begin{equation}
\label{eq:infin_similar}
\lim_{(x,y)\to (a,a)}
\frac{|d_2(x,y)-d_1(x,y)|}{d_1(x,a)+d_1(y,a)}
=0.
\end{equation}
\end{defn}

As usual,
in the definition of
$\lim_{(x,y)\to (a,a)}$
we assume that $(x,y)\ne (a,a)$.
To include the point $(a,a)$,
we rewrite the limit relation~\eqref{eq:infin_similar}
in the following equivalent form:
\begin{equation}
\label{eq:infin_similar2}
\forall\eps>0\quad
\exists V\in\tau(a)\quad
\forall x,y\in V\qquad
|d_2(x,y)-d_1(x,y)|
\le \eps\,\bigl(d_1(x,a)+d_1(y,a)\bigr),
\end{equation}
where $\tau$ is the topology induced by $d_1$ (or by the topologically equivalent distance $d_2$).

\begin{prop}
Let $X$ be a set.
The binary relation $\infsim$,
defined on the set of all distances on $X$,
is an equivalence relation.
\end{prop}

\begin{proof}
Obviously, $\infsim$ is reflexive.

1. Symmetry.
Suppose that $d_1\infsim d_2$.
Let $\tau\eqdef\tau_{d_1}$;
by definition of $\infsim$,
we also have $\tau=\tau_{d_2}$.
Let $a\in X$ and $\eps>0$.
From~\eqref{eq:infin_similar2},
there is $V$ in $\tau(a)$ such that
for every $x,y$ in $V$,
\begin{equation}
\label{eq:application_def_infin}
|d_1(x, y)-d_2(x, y)|
\le
\frac{\min\{\eps,1\}}{2}\,
\bigl(d_1(x, a) + d_1(y, a)\bigr).
\end{equation}
Taking in~\eqref{eq:application_def_infin}
$y = a$ and then $x = a$,  we get
\[
d_1(x, a) - d_2(x, a)
\leq \frac{d_1(x, a)}{2},
\qquad
d_1(y, a) - d_2(y, a)
\leq \frac{d_1(y,a)}{2}.
\]
We sum the last two inequalities and simplify:
\begin{equation}
\label{eq:d1xa_plus_d1ya_le_C_d2xa_d2ya}
d_1(x,a)+d_1(y,a)
\leq
2\bigl(d_2(x,a) + d_2(y,a)\bigr).
\end{equation}
Thus, for every $x,y$ in $V$,
by~\eqref{eq:application_def_infin}
and~\eqref{eq:d1xa_plus_d1ya_le_C_d2xa_d2ya},
\[
|d_2(x,y)-d_1(x,y)|
\le
\frac{\eps}{2}\,
\bigl(d_1(x, a) + d_1(y, a)\bigr)
\le
\eps\,
\bigl(d_2(x, a) + d_2(y, a)\bigr).
\]
2. Transitivity.
Suppose that $d_1\infsim d_2$
and $d_2\infsim d_3$.
Let $\tau\eqdef \tau_{d_1}$.
Then, $\tau=\tau_{d_2}=\tau_{d_3}$.
By symmetry, we have $d_2\infsim d_1$.
Let $a\in X$ and $\eps>0$.
There exists $V$ in $\tau(a)$ such that
for all $x,y,z$ in $V$,
\begin{align*}
\bigl|d_1(x,y)-d_2(x,y)\bigr|
&\le
\frac{\eps}{3}\,
\bigl(d_1(x,a)+d_1(y,a)\bigr),
\\[1ex]
\bigl|d_2(x,y)-d_1(x,y)\bigr|
&\le
\frac{1}{2}\,
\bigl(d_2(x,a)+d_2(y,a)\bigr),
\end{align*}
and
\[
\bigl|d_2(x,y)-d_3(x,y)\bigr|
\le
\frac{\eps}{3}\,
\bigl(d_2(x,a)+d_2(y,a)\bigr).
\]
The second of these inequalities
implies an analog of~\eqref{eq:d1xa_plus_d1ya_le_C_d2xa_d2ya},
with swapped $d_1$ and $d_2$:
\[
d_2(x,a)+d_2(y,a)
\leq
2 \bigl(d_1(x,a) + d_1(y,a)\bigr).
\]
Thus,
\begin{align*}
\bigl|d_1(x,y)-d_3(x,y)\bigr|
&\le
\bigl|d_1(x,y)-d_2(x,y)\bigr|
+\bigl|d_2(x,y)-d_3(x,y)\bigr|
\\[1ex]
&\le
\frac{\eps}{3}\,
\bigl(d_1(x,a) + d_1(y,a)\bigr)
+
\frac{\eps}{3}\,
\bigl(d_2(x,a) + d_2(y,a)\bigr)
\\[1ex]
&\le
\eps\,
\bigl(d_1(x,a) + d_1(y,a)\bigr).
\qedhere
\end{align*}
\end{proof}

\medskip
\subsection*{Comparison of strong local similarity,
infinitesimal similarity,
and local similarity}

\begin{prop}
\label{prop:strsim_implies_infsim}
Let $X$ be a set and $d_1,d_2$
be distances on $X$.
If $d_1\strsim d_2$,
then $d_1\infsim d_2$.
\end{prop}

\begin{proof}
Suppose that $d_1\strsim d_2$.
Let $a\in X$.
We use the following upper bound
for the expression in~\eqref{eq:infin_similar}:
\[
\frac{|d_2(x,y)-d_1(x,y)|}{d_1(x,a)+d_1(y,a)}
\le
\frac{|d_2(x,y)-d_1(x,y)|}{d_1(x,y)}
=
\left|\frac{d_2(x,y)}{d_1(x,y)}-1\right|.
\]
As $(x,y)\to (a,a)$ and $x\ne y$,
the last expression tends to zero.
\end{proof}

\begin{prop}
\label{prop:infsim_implies_locsim}
Let $X$ be a set and $d_1,d_2$
be distances on $X$.
If $d_1\infsim d_2$, then $d_1\cong d_2$.
\end{prop}

\begin{proof}
Suppose that $d_1\infsim d_2$.
Let $a\in X$.
We suppose that the limit relation~\eqref{eq:infin_similar} holds.
Fixing $y=a$ we get the following consequence of~\eqref{eq:infin_similar}:
\[
\lim_{x\to a}
\frac{|d_2(x,a)-d_1(x,a)|}{d_1(x,a)}
=0.
\]
It is equivalent to~\eqref{eq:def_loc_sim}.
\end{proof}

The next example shows, in particular,
that $d_1\infsim d_2$ is not equivalent to $d_1\cong d_2$.

\begin{example}
\label{example:hook}
Let $X$ be the following ``hook'' in $\bR^2$:
\[
X
=
\Bigl([0,+\infty)\times\{0\}\Bigr)
\cup
\Bigl(\{0\}\times[0,+\infty)\Bigr)
=\{(x,0)\colon\ x\ge0\}
\cup
\{(0,y)\colon\ y\ge0\}.
\]
We denote by $d_1$ the ``taxi distance'' on $X$ and by $d_2$ the Euclidean distance on $X$:
\[
d_1((u,v),(r,s))
=|u-r|+|v-s|,
\qquad
d_2((u,v),(r,s))
=\sqrt{(u-r)^2+(v-s)^2}.
\]
For two points belonging to the same halfaxis, we have
\[
d_1((u,0),(r,0))
=|u-r|
=d_2((u,0),(r,0)),
\]
while for two points belonging to the different halfaxes,
\[
d_1((u,0),(0,s))
=u+s,
\qquad
d_2((u,0),(0,s))
=\sqrt{u^2+s^2}.
\]
See Figure~\ref{fig:hook}.

Notice that if $a=(0,0)$, then
$d_2(x,a)=d_1(x,a)$ for every $x$ in $X$.
If $a\ne (0,0)$, then $d_2(x,a)=d_1(x,a)$
for every $x$ sufficiently close to $a$.
Therefore,
$d_1\cong d_2$.

It is also easy to see that
$d_1\lipequiv d_2$
and $d_2^\ast=d_1^\ast=d_1$.

On the other hand,
working with $a=(0,0)$,
$x=(u,0)$, and $y=(0,u)$, we get for every $u>0$:
\[
\frac{d_2((u,0),(0,u))}{d_1((u,0),(0,u))}
=
\frac{\sqrt{2}\,u}{2u}
=
\frac{\sqrt{2}}{2}
\]
and
\[
\frac{|d_2((u,0),(0,u))-d_1((u,0),(0,u))|}{d_1((u,0),(0,0))+d_1((0,u),(0,0))}
=\frac{2-\sqrt{2}}{2}.
\]
Therefore, $d_1$ and $d_2$
are not strongly locally similar
nor infinitesimally similar.
\hfill\qedsymbol
\end{example}

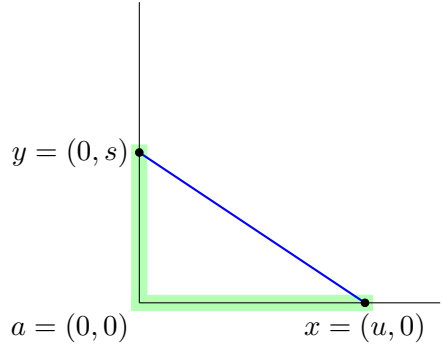
\begin{figure}[htp]
\centering
\begin{tikzpicture}
\filldraw [lightgreen]
  (-0.1, -0.1) rectangle (3.1, 0.1);
\filldraw [lightgreen]
  (-0.1, -0.1) rectangle (0.1, 2.1);
\draw [thick, blue] 
  (3, 0) -- (0, 2);
\draw (0, 0) -- (4, 0);
\draw (0, 0) -- (0, 4);
\node at (0, 0) [below left] {$a=(0,0)$};
\filldraw (3, 0) circle (0.05cm);
\node at (3, 0) [below] {$x=(u,0)$};
\filldraw (0, 2) circle (0.05cm);
\node at (0, 2) [left] {$y=(0,s)$};
\end{tikzpicture}
\caption{Distances from Example~\ref{example:hook}.
\label{fig:hook}}
\end{figure}

Omer Cantor
(\myurl{https://orcid.org/0009-0007-2969-7170})
constructed the following example and explained it to us.
It shows that $d_1\infsim d_2$ is not equivalent to $d_1\strsim d_2$.

\begin{example}
\label{example:OmerCantor}
For each parameter $\al$ satisfying $0<\al\le 1$, we define
$\rho_\al\colon[-1,1]^2\to[0,+\infty)$ by
\[
\rho_\al(x,y)
\eqdef
\begin{cases}
|x-y|,
& xy\ge 0;
\\
|x+y|+\al \bigl(\min\{|x|,|y|\}\bigr)^2,
& xy<0.
\end{cases}
\]
Intuitively, the definition of $\rho_\al$ means that for each $u$ in $(0,1]$ there is a ``teleportation'' of cost $\al u^2$
between $u$ and $-u$.
Thus, if $0<u<v\le 1$, then $\rho_\al(-u,v)$ as obtained as the sum of two terms:
the teleportation between $-u$ and $u$
and the usual walk between $u$ and $v$.
It is easy to verify that $\rho_\al$
is a distance on $[-1,1]$.
It induces the usual topology on $[-1,1]$.
Moreover, if $a \ne 0$, then there exists an open neighborhood $V$ of $a$ such that
$\rho_\al(x,y)=|x-y|$ for each $x,y$ in $V$.
Consider two different distances from this family,
say $\rho_{1/2}$ and $\rho_1$.
For $a=0$ and $x,y\in[-1,1]$ such that $xy<0$, we have
\[
\frac{|\rho_{1/2}(x,y)-\rho_1(x,y)|}{\rho_1(x,0)+\rho_1(0,y)}
=
\frac{\frac{1}{2}\bigl(\min\{|x|,|y|\}\bigr)^2}{|x|+|y|}.
\]
Obviously, this expression tends to $0$ as $(x,y)\to(0,0)$.
For all other choices of $a,x,y$, the corresponding quotient is equal to $0$.
Therefore, $\rho_{1/2}\infsim\rho_1$.

On the other hand, for $a=0$,
$x_k=1/k$, and $y_k=-1/k$,
we consider the quotient
\[
\frac{\rho_{1/2}(x_k,y_k)}{\rho_1(x_k,y_k)}
=
\frac{\frac{1}{2} (1/k)^2}{(1/k)^2}
=
\frac{1}{2}.
\]
As $k\to\infty$, we have $(x_k,y_k)\to(0,0)$,
but the quotient above does not tend to $1$.
Hence, $\rho_{1/2}$ and $\rho_1$ are not strongly locally similar.
\end{example}

\medskip
\section{Locally similar distances and composition}
\label{sec:locally_similar_distances_and_composition}

In this section, we consider a particular situation where two distances $d_1$ and $d_2$ are related by $d_2=f\circ d_1$. We provide a sufficient condition for
$d_2\strsim d_1$ in terms of $f$.
More precisely, for convenience in applications, we provide a sufficient condition for
$d_2\strsim Q d_1$, with $0<Q<+\infty$.
According to Propositions~\ref{prop:strsim_implies_infsim}
and~\ref{prop:infsim_implies_locsim},
our condition is also sufficient for
$d_2\cong Q d_1$.

\begin{lem}
\label{lem:f_cannot_decrease}
Let $f\colon[0,+\infty)\to[0,+\infty)$ and $C>0$ be such that $f(t) \geq C\min\{t, 1\}$ for every $t\ge 0$.
Then
\begin{itemize}
\item[1)] for every $t>0$, $f(t)>0$;
\item[2)] $\displaystyle \lim_{f(t)\to 0^+} t=0$,
i.e.,
\[
\forall\eta>0\quad
\exists\xi>0\quad
\forall t\in[0,+\infty)
\quad
\Bigl(
0<f(t)<\xi
\quad\Longrightarrow\quad
0<t<\eta
\Bigr).
\]
\end{itemize}
\end{lem}

\begin{proof}
Property 1) follows directly from the hypothesis.
We are going to prove 2).
Let $\eta>0$.
Define $\xi\coloneqq \min\{C\eta, C/2 \}$.
Let $t$ be a positive number such that $0<f(t)<\xi$.
Notice that $t<1$, since otherwise we get:
\[
C
= C\min\{t, 1\}
\leq f(t)
<\frac{C}{2}.
\]
Therefore, the result follows:
\[
t = \min\{t, 1\}
\leq \frac{f(t)}{C}
< \frac{\xi}{C}
\le \eta.
\qedhere
\]
\end{proof}

\begin{prop}\label{prop:d_2 = f circ d_1}
Let $(X,d_1)$ be a metric space,
$0<Q<+\infty$,
and $f\colon[0,+\infty)\to[0,+\infty)$ be a function with the following properties:
\begin{enumerate}
\item[1)]
$f(0)=0$;
\item[2)]
$\lim_{t\to 0^+}\frac{f(t)}{t} = Q$;
\item[3)]
there is a constant $C>0$ such that for every $t\geq 0$,
$f(t)\geq C \min\{t, 1\}$.
\end{enumerate}
Define $d_2\eqdef f\circ d_1$.
Suppose that $d_2$ is a distance.
Then,
$d_2\strsim Q d_1$
and
$d_2 \cong Q d_1$.
\end{prop}

\begin{proof}
Part 1. Let us verify an analog of~\eqref{eq:str_sim_first_limit} for $Q d_1$ and $d_2$.
Given $\eps>0$,
by 2), there is $\de_1>0$ such that, if $0<t<\de_1$, then
\[
\left|\frac{f(t)}{t}- Q\right|<Q\eps .
\]
Define
$\de_2\eqdef \frac{Q \de_1}{2}$.
Let $a\in X$ and $x,y\in X$ such that
\[
x \ne y,\qquad
Q d_1(x, a)<\de_2,\qquad
Q d_1(y, a)<\de_2.
\]
Put $s \eqdef d_1(x,y)$.
Then $0<s \leq d_1(x, a) + d_1(a, y)<2\de_2/Q = \de_1$ and
\[
\left|\frac{d_2(x,y)}{Qd_1(x,y)}-1\right|=\frac{1}{Q}
\left|\frac{d_2(x,y)}{d_1(x,y)}-Q\right|
=\frac{1}{Q}
\left|\frac{f(s)}{s}-Q\right|
<\eps.
\]
Part 2. Let us verify an analog of~\eqref{eq:str_sim_second_limit} for $Q d_1$ and $d_2$.
By property 2),
\[
\lim_{t\to 0^+}
\frac{t}{f(t)} = \frac{1}{Q}.
\]
Given $\eps>0$,
there is $\de_3>0$ such that
if $0<t<\de_3$, then
\[
\left|\frac{t}{f(t)}
-\frac{1}{Q}\right|
<\frac{\eps}{Q}.
\]
From Proposition~\ref{lem:f_cannot_decrease}, there is $\de_4>0$ such that if $0<f(t)<\de_4$, then $0<t<\de_3$. 

Let $a\in X$ and $x, y\in X$ such that $x\neq y$,
$d_2(x, a)<\de_4/2$,
and $d_2(y, a)<\de_4/2$.
Put $s\eqdef d_1(x,y)$.
Then $0< f(s)=d_2(x,y)\leq d_2(x, a) + d_2(a, y) < \de_4$.
Therefore,
$0<s<\de_3$, and
\[
\left|\frac{Q d_1(x,y)}{d_2(x,y)}
-1\right|
=
Q
\left|\frac{d_1(x,y)}{d_2(x,y)}
-\frac{1}{Q}\right|
=Q\left|\frac{s}{f(s)}-\frac{1}{Q} \right| < \eps.
\qedhere
\]
\end{proof}

\begin{cor}
\label{cor:f_concave}
Let $(X,d_1)$ be a metric space,
$0<Q<+\infty$,
and $f\colon[0,+\infty)\to[0,+\infty)$ be a concave increasing function such that $f(0)=0$ and
\[
\lim_{t\to 0^+}\frac{f(t)}{t} = Q.
\]
Define $d_2\eqdef f\circ d_1$.
Then,
$d_2\strsim Q d_1$
and
$d_2\cong Q d_1$.
\end{cor}

\begin{proof}
The assumptions on $f$ imply that $f(t)>0$ for $t>0$.
Moreover, $d_2$ is a distance because 
$f$ is concave and increasing,
and $f(0)=0$
(see a similar fact in
\cite[Section~4.1]{deza_encyclopedia}).
We will verify property 3) from Proposition~\ref{prop:d_2 = f circ d_1}.
Let $t\ge 0$. 
If $t<1$, then
\[
tf(1) 
= 
(1-t)f(0) + tf(1) 
\leq 
f(t).
\]
If $t\geq 1$, then $f(1)\leq f(t)$,
because $f$ is increasing.
Joining these two cases,
for every $t\geq 0$ we get
\[
f(1)\min\{t, 1\} \leq f(t). \qedhere
\]
\end{proof}

\section{The intrinsic distance induced by a given distance}
\label{sec:intrinsic_distance}

In this section,
we recall the definition of the intrinsic distance $d^\ast$ induced by a given distance $d$.
It is a standard construction in the theory of length spaces (see, for example, Gromov~\cite[Example 1.4 (a)]{gromov_riemannian},
Burago, Burago, and Ivanov~\cite[Section~2.3]{burago_metric},
Kunzinger and Steinbauer~\cite[Section~1.2.3]{kunzinger_alexandrov}).
In this section, we also compute $d^\ast$ for several examples.
In particular, we show that $\lipequiv$ does not imply $\intrequiv$
and $\intrequiv$ does not imply $\topequiv$.

We now fix some notation
for the intermediate concepts
which will be used in the examples and the proof of the main result.

\begin{defn}[continuous paths between two points]
\label{def:continuous_paths}
Let $(X, d)$ be a metric space. 
For each $x, y$ in $X$, define
\[
\Ga_d(x, y) 
\eqdef
\bigl\{
\ga\in C([0,1], X, d)\colon \quad
\ga(0) = x,\quad
\ga(1) = y
\bigr\}.
\]
Here $C([0,1],X,d)$ stands for the set of all functions
$\ga\colon[0,1]\to X$ that are continuous with respect to $\tau_d$.
\end{defn}

\begin{defn}[Riemannian partitions of an interval]
Let $[a, b]\subseteq \bR$ be an arbitrary interval.
We denote by $\cP([a, b])$ the set of all finite subsets of $[a,b]$ that include $a$ and $b$.
In particular,
we denote $\cP([0, 1])$ by $\cP$.
\end{defn}

Given $P$ in $\cP([a, b])$, we enumerate its elements in ascending order:
\begin{equation}
\label{eq:ordered_partition}
P = \{t_0, t_1, \ldots, t_m\}, 
\qquad
a = t_0<t_1<\cdots<t_m =b.
\end{equation}

\begin{defn}[the polygonal length of a path with respect to a partition]
Let $(X, d)$ be a metric space,
$x, y\in X$, $\ga\in \Ga_d(x, y)$,
and $P\in\cP$ be of the form~\eqref{eq:ordered_partition}.
Define
\[
S(d, \ga, P) 
\eqdef
\sum_{k=0}^{m-1}
d(\ga(t_{k}), \ga(t_{k+1})).
\]
\end{defn}

\begin{lem}\label{lem:polygonal_partition_is_increasing_wrt_contention}
Let $(X, d)$ be a metric space,
$x, y\in X$, and $\ga\in\Ga_d(x, y)$. 
If $P, Q\in\cP$
such that $P\subseteq Q$, then
\[
S(d, \ga, P) \leq S(d, \gamma, Q).
\]
\end{lem}

\begin{proof}
It follows by an iterative application of the triangle inequality.
\end{proof}

\begin{defn}[the length of a path with respect to a distance]
\label{def:length_of_gamma_wrt_d}
Let $(X, d)$ be a metric space,
$x, y\in X$,
and $\ga\in \Ga_d(x, y)$.
Define
\[
L(d, \ga)
\eqdef
\sup_{P\in \cP}
S(d, \ga, P).
\]
\end{defn}

\begin{defn}[the intrinsic distance induced by a distance]
\label{def:intrinsic_distance}
Let $(X, d)$ be a metric space.
Define
$d^\ast\colon X^2\to [0,+\infty]$
by the following rule:
\[
d^\ast(x, y)
\eqdef
\inf_{\ga\in \Ga_d(x, y)}
L(d, \ga).
\]
\end{defn}

Following Definition~\ref{def:intrinsic_distance},
if $\Ga_d(x, y)=\emptyset$
(i.e., if there are no continuous paths between $x$ and $y$),
then $d^\ast(x,y)=+\infty$.

It is well known
that
$d^\ast$ is a $[0,+\infty]$-valued
distance on $X$,
such that $d^\ast(x,y)\ge d(x,y)$ for every $x,y$ in $X$.
Moreover, it is well known that $(d^\ast)^\ast=d^\ast$.
These facts are proved in the books cited
at the beginning of this section.

\begin{rem}
\label{rem:general_intervals_instead_of_unit_interval}
In this paper, we define paths as continuous functions with domain $[0, 1]$. 
Similarly, one can consider paths defined on arbitrary closed intervals of the form $[a, b]$. 
Such generalization does not change $L(d, \gamma)$ or $d^\ast(x, y)$, 
since any path defined on an arbitrary interval $[a, b]$ can be reparametrized to $[0, 1]$.
\end{rem}

\begin{prop}
\label{prop:lipequiv_implies_lipequiv_for_intrinsic_distances}
Let $d_1,d_2$ be distances on $X$ such that $d_1\lipequiv d_2$.
Then $d_1^\ast\lipequiv d_2^\ast$.
\end{prop}

\begin{proof}
The proof is a simple exercise on supremums and infimums;
we give it for the sake of completeness.
By the definition of $\lipequiv$, there exist $C_1,C_2$ in $(0,+\infty)$ such that
$d_1\le C_1 d_2$
and $d_2 \le C_2 d_1$.
Notice that $\Ga_{d_1}(x,y)=\Ga_{d_2}(x,y)$ for every $x,y$ in $X$, because $d_1\topequiv d_2$.

Let $x,y\in X$
and $\ga\in\Ga_{d_1}(x,y)$.
For every $P$ in $\cP$,
\[
S(d_2,\ga,P)
\le C_2 S(d_1,\ga,P)
\le C_2 L(d_1, \ga).
\]
Passing to the supremum over $P$ we conclude that
$L(d_2,\ga)\le C_2 L(d_1,\ga)$.
So, for every $\ga$ in $\Ga_{d_1}(x,y)$,
\[
\frac{1}{C_2} d_2^\ast(x,y)
\le
\frac{1}{C_2} L(d_2,\ga)
\le L(d_1,\ga).
\]
Passing to the infimum over $\ga$ in $\Ga_{d_1}(x,y)$
we get $d_2^\ast(x,y) \le C_2 d_1^\ast(x,y)$.
In a similar manner,
$d_1^\ast\le C_1 d_2^\ast$.
\end{proof}

\begin{example}[Euclidean distance on $\bR^n$]
\label{example:Euclidean_distance}
    We denote by $d_{\bR^n}$ the usual Euclidean distance on $\bR^n$.
    It is well known
    and easy to verify that $d_{\bR^n}^\ast = d_{\bR^n}$.   
    Identifying $\bC^n$ with $\bR^{2n}$, we obtain $d^\ast_{\bC^n} = d_{\bC^n}$.
    \hfill\qedsymbol
\end{example}

\begin{example}[the intrinsic distance induced by the finite discrete distance]
\label{example:d0}
Let $X$ be a set having more than one point and
$d_0\colon X^2\to[0,+\infty)$
is defined by
\[
d_0(x,y)
\eqdef
\begin{cases}
0, & x=y; \\
1, & x\ne y.
\end{cases}
\]
Since $d_0$ induces the discrete topology,
there are no continuous paths between different points of $X$,
and $d_0^\ast(x, y)=+\infty$ if $x\neq y$.
\hfill\qedsymbol
\end{example}

\begin{example}[square-root distance on $\bR$]
\label{example:sqrt}
Consider the distance
$\dsqrt\colon\bR^2\to[0,+\infty)$, 
given by
$\dsqrt(x, y)\eqdef\sqrt{|x-y|}$.
The topology induced by $\dsqrt$ is the usual topology in $\bR$, 
since the set of open balls with respect to the Euclidean distance on $\bR$ is the same as the set of open balls
with respect to the distance $\dsqrt$. 

Let $x,y\in\bR$. We will show that if $x\neq y$, then
$\dsqrt^\ast(x, y)=+\infty$. 
Suppose $x<y$ (the case if $x>y$ is similar).
Let $\ga\in\Ga_{\dsqrt}(x,y)$.
Notice that $\ga$ is continuous in the usual sense.
Let $m\in\bN$. For every $j$ in $\{0,\ldots, m\}$, we define
\[
z_{m,j} \eqdef x+\frac{j}{m}(y-x).
\]
By the intermediate value theorem, there is $t$ in $[0,1]$ such that $\gamma(t) = z_{m,j}$. 
Define
\[
t_{m,m}\eqdef 1,
\]
\[
t_{m,j}\eqdef
\inf\,
\bigl\{
t\in[0,1]\colon\ 
\ga(t) = z_{m,j}
\bigr\}
\qquad (j\in\{0,\ldots,m-1\}).
\]
Because $\ga$ is continuous,
$\ga(t_{m,j}) = z_{m,j}$.
Furthermore, the intermediate value theorem implies 
\[
t_{m,0}<t_{m,1}<\cdots <t_{m,m}.
\]
Using the partition 
$P_m \coloneqq\{t_{m,0}, t_{m,1}, \dots, t_{m,m}\}$,
we establish the following lower bound for $L(\dsqrt, \gamma)$:
\begin{align*}
    L(\dsqrt, \gamma)
    &\geq 
    S(\dsqrt, \gamma, P_m) 
    = 
    \sum_{j =0}^{m-1} \dsqrt(\gamma(t_{m,j}), \gamma(t_{m,j+1}))\\
    &= \sum_{j =0}^{m-1} \sqrt{z_{m,j+1}- z_{m,j}}
    = 
    \sum_{j =0}^{m-1} \sqrt{\frac{1}{m}(y-x)}
    = 
    \sqrt{m(y-x)}.
    \end{align*}
Taking the supremum over $m$ in $\bN$ shows that $L(\dsqrt, \ga) = +\infty$.
Since $\ga$ was an arbitrary path between $x$ and $y$,
it follows from Definition~\ref{def:intrinsic_distance}, that
$\dsqrt^\ast (x, y) = +\infty$ if $x\neq y$.

Let $d_0\colon\bR^2\to[0,+\infty)$
is defined as in Example~\ref{example:d0}.
Although $d_0$ and $\dsqrt$
are not topologically equivalent, 
their corresponding intrinsic distances are, in fact, equal:
$d_0^\ast = \dsqrt^\ast$.
\hfill\qedsymbol
\end{example}

In the following example we develop the idea from Example~\ref{example:sqrt}
in such a manner that the induced intrinsic distance is finite.

\begin{example}[quadrant of the plane with the usual distance added with the sqrt-angular pseudodistance]
\label{example:quadrant_angular}
Let $X$ be the closed top right quadrant
of the closed unit disk on the complex plane:
\[
X
\eqdef
\bigl\{z\in\bC\colon\ |z|\le 1,\ \Re(z)\ge 0,\ \Im(z)\ge 0\bigr\}.
\]
Given $z$ in $X$, let $\arg(z)$ be the principal argument of $z$ belonging to $[0,\pi/2]$.
We define $d\colon X^2\to[0,+\infty)$ by
\[
d(z,w)
\eqdef
\begin{cases}
|z|+|w|, & zw=0;
\\[1ex]
|z-w|, & zw\ne 0,\ \arg(z)=\arg(w);
\\[1ex]
|z-w| + \sqrt{|\arg(z)-\arg(w)|}, &
zw\ne 0,\ \arg(z)\ne\arg(w);
\end{cases}
\]
Furthermore, we define
$\rho\colon X^2\to[0,+\infty)$ by
\[
\rho(z,w)
\eqdef
\begin{cases}
|z-w|, & zw\ne 0,\ \arg(z)=\arg(w);
\\[1ex]
|z|+|w|, & \text{otherwise}.
\end{cases}
\]
Both $d$ and $\rho$ are distances on $X$.
The topology induced by $d$
coincides with the usual topology of the set $X$ induced by the Euclidean distance.
The topology induced by $\rho$ is different.
Indeed, if $z\ne0$ and $0<\de<\min\{|z|,1-|z|\}$,
then the ball $B_\rho(z,\de)$
is in fact the segment
$\{\la z\colon\ |z|-\de<\la<|z|+\de\}$.

Given $z,w$ in $X$ such that $z\ne0$, $w\ne0$, and $\arg(z)\ne\arg(w)$,
and a path $\ga\in \Ga_d(z, w)$ such that $0\notin \ga([0,1])$,
we apply the reasoning from Example~\ref{example:sqrt} to $\arg\circ\ga$
and conclude that
$L(d,\ga)=+\infty$.

Therefore, if $\eta$ is a path of finite length from $z$ to $w$, then it must pass through $0$.
Moreover, some subpaths of $\eta$
must be contained in the segments joining $z$ with $0$ and $0$ with $w$.
Figure~\ref{fig:quadrant_angular}
shows a couple of paths:
one has finite length with respect to $d$
and the other has infinite length.

It follows that $d^\ast=\rho$.
Consequently, $\rho^\ast=\rho$.
In this example,
$d$ is not topologically equivalent to $\rho$,
but they induce the same intrinsic distance.
\hfill\qedhere
\end{example}

\begin{figure}[htb]
\centering
\begin{tikzpicture}
\filldraw [lightgreen]
  (0, 0) -- (4, 0) arc (0: 90: 4cm) -- cycle;
\draw [darkgreen]
  (0, 0) -- (4, 0) arc (0: 90: 4cm) -- cycle;
\draw [blue, very thick]
  (20 : 3cm) -- (0, 0) -- (80 : 2cm);
\draw [red, very thick]
  plot [smooth] coordinates
  { (2.81907, 1.02606)
    (3, 1.5)
    (2.7, 1.8)
    (1.9, 1.6)
    (1.6, 2.1)
    (0.8, 2.2)
    (0.34730, 1.96962)
  };
\filldraw [blue]  (20 : 3cm) circle (0.6mm);
\node at (20 : 3cm) [right] {$z$};
\filldraw [blue]  (80 : 2cm) circle (0.6mm);
\node at (80 : 2cm) [above] {$w$};
\node at (4, 0) [below] {$1$};
\node at (0, 4) [left] {$\iu$};
\node at (0, 0) [below left] {$0$};
\end{tikzpicture}
\caption{The set $X$
from Example~\ref{example:quadrant_angular},
two points belonging to $X$ ($z$ and $w$),
a path of finite length between them (blue)
and a path of infinite length (red),
with respect to $d$.
\label{fig:quadrant_angular}}
\end{figure}

\begin{example}[pseudologarithmic distance on the halfline]
\label{example:pseudologarithmic_distance_on_halfline}
Let $X \eqdef (0,+\infty)$. Define $\rho\colon X^2\to[0,+\infty)$,
\[
\rho(x,y)\eqdef
\frac{2|x-y|}{x+y}.
\]
We are going to prove that
\begin{equation}
\label{eq:intrinsic_distance_for_pseudologarithmic_distance}
\rho^\ast(x,y)
=|\log(x)-\log(y)|.
\end{equation}
Consider the case where $x<y$
(the case $x>y$ is similar).

1. Let
$\xi(t)\eqdef x+t(y-x)$
be the straight-line path between $x$ and $y$, and let $Q\in\cP([0,1])$.
For each $m$ in $\bN$,
we add to $Q$ the points
$k/m$, where $k\in\{0,1,\ldots,m\}$,
and obtain a refined partition
$R_m=\{t_{m,0},t_{m,1},\ldots,t_{m,m}\}$.
For each $k$ in $\{0, 1, \ldots, m\}$, we define
$z_{m,k}\eqdef\xi(t_{m,k})$. 
This setup leads to an explicit formula for the polygonal length of
$\xi$ associated with $R_m$:
\begin{align*}
S(\rho,\xi,Q)
&\le
S(\rho,\xi,R_m)
=
\sum_{k=0}^{m-1}
\rho(\xi(t_{m,k}),\xi(t_{m,k+1}))
=
\sum_{k=0}^{m-1}
\rho(z_{m,k},z_{m,k+1})
\\[1ex]
&=
\sum_{k=0}^{m-1}
\frac{2(z_{m,k+1}-z_{m,k})}{z_{m,k+1}+z_{m,k}}
=
\sum_{k=0}^{m-1}
\frac{z_{m,k+1}-z_{m,k}}{\frac{z_{m,k+1}+z_{m,k}}{2}}.
\end{align*}
The resulting sum is a Riemannian sum for the integral of the function
$f(t) = \frac{1}{t}$ over the interval $[x, y]$.
The points $z_{m,0},z_{m,1},\ldots,z_{m,m}$ form a Riemannian partition of $[x,y]$,
and its mesh tends to $0$
as $m\to\infty$:
\[
\min_{k\in\{0,\ldots,m-1\}}
\bigl(z_{m,k+1}-z_{m,k}\bigr)
=
(y-x)\,
\min_{k\in\{0,\ldots,m-1\}}
\bigl(t_{m,k+1}-t_{m,k}\bigr)
\le
\frac{y-x}{m}.
\]
Therefore,
\[
S(\rho,\xi,Q)
\le
\lim_{m\to\infty}
S(\rho,\xi,R_m)
=\int_x^y \frac{\dif{}t}{t}
=\log(y)-\log(x).
\]
Since $Q$ was an arbitrary element of $\cP([0,1])$,
we have proven that
$L(\rho,\xi)\le \log(y)-\log(x)$.

2. For an arbitrary $\ga$ in $\Ga_{\rho}(x,y)$,
we can construct a partition
$P_m=\{t_{m,0},\ldots,t_{m,m}\}$
with $\ga(t_{m,k})=x+k(y-x)/m$,
as in Example~\ref{example:sqrt},
and obtain
\[
L(\rho,\ga)
\ge S(\rho,\ga,P_m)
= \sum_{k=0}^{m-1}
\rho\left(x+k\frac{y-x}{m},
x+(k+1)\frac{y-x}{m}\right).
\]
The right-hand side is also a Riemannian sum for the integral
$\int_x^y \dif{}t/t$.
Notice that the points
$\ga(t_{m,k})$,
where $k\in\{0,\ldots,m\}$,
form a Riemannian partition of $[x,y]$,
and the mesh of this partition equals $(y-x)/m$.
Passing to the limit as $m\to\infty$ we get
\[
L(\rho,\ga)
\ge
\int_x^y \frac{\dif{}t}{t}
=
\log(y)-\log(x).
\]
Thereby, we have proven~\eqref{eq:intrinsic_distance_for_pseudologarithmic_distance}.
\hfill\qedsymbol
\end{example}

\begin{example}[pseudologarithmic distance on a segment]
\label{example:pseudologarithmic_distance_on_segment}
Let $X \coloneqq [1, 2]$.
Define
$\rho\colon X^2\to[0,+\infty)$ as
\[
\rho(x, y) \eqdef \frac{2|x-y|}{x+y}.
\]
In other words, we consider a metric subspace of the metric space from Example~\ref{example:pseudologarithmic_distance_on_halfline}.
It is easy to see that
$\rho^\ast(x,y)
=|\log(x)-\log(y)|$.
Let $d_{[1,2]}$ be the Euclidean distance on $[1,2]$.
It is easy to see that
$\rho\lipequiv d_{[1,2]}$,
but $\rho^\ast\ne d_{[1,2]}=d_{[1,2]}^\ast$.
Moreover, there is no $C>0$ such that $\rho^\ast=C d_{[1,2]}^\ast$.
\hfill\qedsymbol
\end{example}

\begin{example}[pseudohyperbolic distance on the halfplane]
\label{example:pseudohyperbolic_halfplane}
Let $\bH$ be the upper halfplane of the complex plane:
\[
\bH\eqdef\{z\in\bC\colon\ \Im(z)>0\}.
\]
Define $\rho_\bH\colon \bH^2\to [0,+\infty)$,
\begin{equation}\label{eq:def_pseudohyperbolic_distance_on_halfplane}
\rho_\bH(z, w)
\eqdef 
\left|\frac{w-z}{w-\overline{z}}\right|. 
\end{equation}
For numbers on the imaginary axis, $\rho_\bH$ is the distance from Example~\ref{example:pseudologarithmic_distance_on_halfline}, divided by 2:
\[
\rho_{\bH}(ia,ib)
=\frac{|a-b|}{a+b}
\qquad(a,b>0).
\]
The distance $\rho_{\bH}$ between two points is always greater or equal to the distance $\rho_{\bH}$ between their projections to the imaginary axis:
\begin{equation}
\label{eq:hyp_dist_and_projections}
\rho_{\bH}(z,w)
\ge
\rho_\bH(i\, \Im(z), i\, \Im(w)).
\end{equation}
Indeed, for $z=x+iy$ and $w=u+iv$,
\begin{align*}
&\rho_\bH(z, w)^2 -\rho_\bH(i\, \Im(z), i\, \Im(w))^2 
=
\frac{(x-u)^2+(y-v)^2}{(x-u)^2+(y+v)^2} - \frac{(y-v)^2}{(y+v)^2}
\\[1ex]
&\quad=
\frac{(x-u)^2(y+v)^2-(x-u)^2(y-v)^2}{((x-u)^2+(y+v)^2)(y+v)^2}
=
\frac{4vy(x-u)^2}{((x-u)^2+(y+v)^2)(y+v)^2}
\geq 0.
\end{align*}
Using~\eqref{eq:hyp_dist_and_projections} it is easy to see that the straight path between $ia$ and $ib$ is shorter than any other path. Therefore, by~\eqref{eq:intrinsic_distance_for_pseudologarithmic_distance},
\[
\rho_\bH^\ast(ia, ib) 
= 
\frac{1}{2}|\log(a)-\log(b)| 
\qquad
(a,b >0).
\]
Given arbitrary $z,w$ in $\bH$, it is possible to find a M\"{o}bius transform $\phi$ of $\bH$ and some numbers $a,b>0$ such that $\phi(z)=ia$, $\phi(w)=ib$;
see, e.g., \cite[Section~3.5]{anderson_hyperbolic_geometry}.
Since $\rho_\bH$ is invariant under the M\"{o}bius transforms of $\bH$,
$\rho_{\bH}^\ast(z,w)=\rho_{\bH}^\ast(ia,ib)$.
After an explicit computation of
$\phi$, $a$, and $b$,
one obtains
\begin{equation}
\label{eq:hyperbolic_distance_halfplane}
\rho_{\bH}^\ast(z,w)
=
\frac{1}{2}
\log\frac{1+\rho_{\bH}(z,w)}{1-\rho_{\bH}(z,w)}
=\frac{1}{2}
\log\frac{|w-\overline{z}|+|w-z|}{|w-\overline{z}|-|w-z|}.
\end{equation}
\hfill\qedsymbol
\end{example}

\begin{example}[pseudohyperbolic distance on the unit disk]
\label{example:pseudohyperbolic_disk}
Let $\bD\eqdef\{z\in\bC\colon\ |z|<1\}$.
Define $\rho_\bD\colon \bD^2\to [0,+\infty)$,
\begin{equation}\label{eq:def_pseudohyperbolic_distance_on_disk}
\rho_\bD(z, w)
\eqdef 
\left|\frac{w-z}{1-w\overline{z}}\right|. 
\end{equation}
A Cayley transform between $\bH$ and $\bD$ converts $\rho_\bH$ into $\rho_\bD$.
Thereby, it is possible to show that
\begin{equation}
\label{eq:hyperbolic_distance_disk}
\rho_\bD^\ast(z,w)
=\frac{1}{2}
\log\frac{1+\rho_{\bD}(z,w)}{1-\rho_{\bD}(z,w)}
=\frac{1}{2}
\log\frac{|1-w\overline{z}|+|w-z|}{|1-w\overline{z}|-|w-z|}.
\end{equation}
\hfill\qedsymbol
\end{example}

It is well known
(see, e.g., \cite[Chapter~33]{Voight2021}
and \cite[Section~1.5]{Zhu2005})
that $2\rho_{\bH}^\ast$
is the geodesic distance associated with the Poincar\'{e} halfplane metric
$\dif{}x\,\dif{}y/y^2$,
and $2\rho_{\bD}^\ast$
is the geodesic distance associated with the Poincar\'{e} disc metric
$\dif{}x\,\dif{}y/(1-x^2-y^2)^2$.
Nevertheless, we emphasize that the intrinsic distances
$\rho_{\bH}^\ast$ and $\rho_{\bD}^\ast$
in these classical examples
can be computed without using Riemannian metrics.

\section{Locally similar distances induce the same intrinsic distance}
\label{sec:main_result}

In this section, we prove the main result of the paper (Theorem~\ref{thm:main}).

We start with the following elementary lemma
which provides a sufficient condition for 
one interval to be a subset of another.
Figure~\ref{fig:interval_inclusion} illustrates such intervals.
Lemma~\ref{lem:interval_inclusion}
is useful in the proof of Lemma~\ref{lem:finite_cover}.

\begin{lem}
\label{lem:interval_inclusion}
Let $x,y\in\bR$ and $r,s>0$ such that one of the following conditions holds:
\begin{enumerate}
\item[(a)] $y<x$ and $x+r<y+s$;
\item[(b)] $x<y$ and $y-s<x-r$.
\end{enumerate}
Then, $(x-r,x+r)\subseteq (y-s,y+s)$.
\end{lem}

\begin{proof}
We only consider the case (a).
First, we compare the radii of the intervals:
\[
s > (x - y) + r > r.
\]
Then, we make a conclusion
about the left extremes:
$y - s < x - r$.
\end{proof}

\begin{figure}[htb]
\centering
\begin{tikzpicture}
\fill [fill=lightblue]
  (-1.5, 0) rectangle (3.5, 0.2);
\fill [fill=lightgreen]
  (-4, -0.2) rectangle (4, 0);
\draw (-5, 0) -- (5, 0);
\draw (-4, -0.2) -- (-4, 0);
\node at (-4, -0.2) [below]
  {$\scriptstyle y-s$};
\draw (4, -0.2) -- (4, 0);
\node at (4, -0.2) [below]
  {$\scriptstyle y+s$};
\draw (0, -0.2) -- (0, 0);
\node at (0, -0.2) [below]
  {$\scriptstyle y$};
\draw (-1.5, 0) -- (-1.5, 0.2);
\node at (-1.5, 0.2) [above]
  {$\scriptstyle x-r$};
\draw (3.5, 0) -- (3.5, 0.2);
\node at (3.5, 0.2) [above]
  {$\scriptstyle x+r$};
\draw (1, 0) -- (1, 0.2);
\node at (1, 0.2) [above]
  {$\scriptstyle x$};
\end{tikzpicture}
\\[1ex]
\begin{tikzpicture}
\fill [fill=lightblue]
  (-3.5, 0) rectangle (-0.5, 0.2);
\fill [fill=lightgreen]
  (-4, -0.2) rectangle (4, 0);
\draw (-5, 0) -- (5, 0);
\draw (-4, -0.2) -- (-4, 0);
\node at (-4, -0.2) [below]
  {$\scriptstyle y-s$};
\draw (4, -0.2) -- (4, 0);
\node at (4, -0.2) [below]
  {$\scriptstyle y+s$};
\draw (0, -0.2) -- (0, 0);
\node at (0, -0.2) [below]
  {$\scriptstyle y$};
\draw (-2, 0) -- (-2, 0.2);
\node at (-2, 0.2) [above]
  {$\scriptstyle x$};
  \node at (-3.5, 0.2) [above]
  {$\scriptstyle x-r$};
\draw (-0.5, 0.2) -- (-0.5, 0);
\node at (-0.5, 0.2) [above]
  {$\scriptstyle x+r$};
\draw (-3.5, 0.2) -- (-3.5, 0);
\end{tikzpicture}
\caption{Lemma~\ref{lem:interval_inclusion},
case (a) (above)
and (b) (below).
\label{fig:interval_inclusion}}
\end{figure}
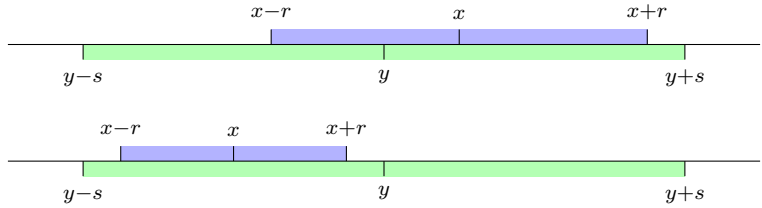

Next, we need the following fact about open covers
and Riemannian partitions of $[a,b]$.

\begin{lem}
\label{lem:finite_cover}
Let $a,b\in\bR$, $a<b$,
and let $R\colon[a,b]\to(0,+\infty)$.
Then, there exist $m$ in $\bN$
and $t_0,\ldots,t_m$ in $[0,1]$
such that
$a=t_0<t_1<\cdots<t_m=b$
and
\begin{equation}
\label{eq:difference_t_less_max_R}
t_k-t_{k-1}
<\max\{R(t_{k-1}), R(t_k)\}
\qquad(1\le k\le m).
\end{equation}
\end{lem}

\begin{proof}
Let $\cQ$ be the subset of $\cP([a,b])$
such that the corresponding $R(x)/2$-neighborhoods of its elements cover $[a,b]$:
\[
\cQ
\eqdef
\left\{P\in\cP([a,b])\colon\quad
[a,b]\subseteq
\bigcup_{x\in P}
\left(x-\frac{R(x)}{2},x+\frac{R(x)}{2}\right)
\right\}.
\]
By the Heine--Borel theorem,
$\cQ\ne\emptyset$.
Assigning to each set $P$ belonging to $\cQ$
its cardinality $\# P$,
we get a nonempty set 
$\{\#P\colon\ P\in\cQ\}$
of natural numbers.
By the induction principle,
it has a minimal element.
We denote it by $m+1$:
\[
m \eqdef \min\bigl(\{\# P\colon\ P\in\cQ\}\bigr) - 1.
\]
Then, we choose in $\cQ$
an element $Q$ of minimal size $m+1$.

Let $Q=\{t_0,t_1,\ldots,t_m\}$,
where $a=t_0<t_1<\ldots<t_m=b$.
For brevity, we set $r_k\eqdef R(t_k)/2$ for each $k$.
We claim that for every $k$ in $\{1,\ldots,m\}$,
\begin{equation}
\label{eq:t_plus_r_greater_t_minus_r}
t_{k-1}+r_{k-1}>t_k-r_k.
\end{equation}
In other words, we consider the list of intervals
$((t_k-r_k,t_k+r_k))_{k=0}^m$,
and we are going to prove that
each pair of consecutive intervals from this list intersect.

Reasoning by reductio ad absurdum
we suppose that 
$k$ is an element of $\{1,\ldots,m\}$
such that
\begin{equation}
\label{eq:wrong_inequality_between_the_neighbors}
t_{k-1}+r_{k-1}\le t_k-r_k.
\end{equation}
Since $t_k-r_k\in[a,b]$, the following inequalities hold:
\[
a \le t_{k-1}
< t_{k-1}+r_{k-1}
\le t_k-r_k
< t_k
\le b.
\]
Since $Q\in\cQ$,
the point $t_k-r_k$ must be covered
by some of the intervals
$\left(x-\frac{R(x)}{2},x+\frac{R(x)}{2}\right)$, where $x\in Q$.
We choose $j$ in $\{0,1,\ldots,m\}$
such that
\begin{equation}
\label{eq:pointk_belongs_intervalj}
t_j-r_j < t_k-r_k < t_j+r_j.
\end{equation}
Obviously, $j\ne k$.
Moreover, due to~\eqref{eq:wrong_inequality_between_the_neighbors}, $j\ne k-1$.

First, consider the case where $j>k$.
Then, $t_j>t_k$ and $t_j-r_j<t_k-r_k$.
By Lemma~\ref{lem:interval_inclusion},
\[
(t_k-r_k,t_k+r_k)
\subseteq
(t_j-r_j,t_j+r_j).
\]
Notice that $k\in\{1,\ldots,m-1\}$
because $1\le k<j\le m$.
Therefore, we can exclude $\{t_k\}$
from $Q$ without losing the covering property
nor the property that $a,b\in Q$.
So, $Q\setminus \{t_k\}$
belongs to $\cQ$ and has fewer elements than $Q$,
which contradicts the minimality of $Q$.

Now, consider the case where $j<k-1$.
In this case, $t_j < t_{k-1}$.
On the other hand,
by~\eqref{eq:pointk_belongs_intervalj}
and~\eqref{eq:wrong_inequality_between_the_neighbors},
\[
t_{k-1}+r_{k-1}
\le t_k - r_k
< t_j + r_j,
\]
By Lemma~\ref{lem:interval_inclusion},
\[
(t_{k-1}-r_{k-1},t_{k-1}+r_{k-1})
\subseteq
(t_j-r_j,t_j+r_j).
\]
Notice that $k-1\in\{1,\ldots,m-1\}$
because $0\le j < k-1 < k\le m$.
Excluding $t_{k-1}$ from $Q$
we get an element of $\cQ$
with a smaller number of elements,
which again contradicts the minimality of $Q$.

Therefore,~\eqref{eq:t_plus_r_greater_t_minus_r} is true,
and \eqref{eq:t_plus_r_greater_t_minus_r} easily implies~\eqref{eq:difference_t_less_max_R}:
\[
t_k - t_{k-1}
< r_{k-1}+r_k
= \frac{R(t_{k-1})+R(t_k)}{2}
\le \max\{R(t_{k-1}), R(t_k)\}.
\qedhere
\]
\end{proof}

\begin{rem}
\label{rem:one_point_is_center_for_the_other}
The property~\eqref{eq:difference_t_less_max_R} 
from Lemma~\ref{lem:finite_cover}
has a simple geometrical meaning:
for each $k$,
at least one of the points $t_{k-1}$ and $t_k$ belongs to the interval associated with the other point.
In other words,
for each $k$ in $\{1,\ldots,m\}$,
at least one of the following conditions is fulfilled:
\begin{itemize}
\item[(a)] $t_k\in(t_{k-1}-R(t_{k-1}),t_{k-1}+R(t_{k-1}))$;
\item[(b)] $t_{k-1}\in (t_k-R(t_k),t_k+R(t_k))$.
\end{itemize}
Indeed, if
$\max\{R(t_{k-1}),R(t_k)\}=R(t_{k-1})$,
then condition (a) holds.
Otherwise,
we get
$\max\{R(t_{k-1}),R(t_k)\}=R(t_k)$,
and (b) holds.
\end{rem}

\begin{rem}
\label{rem:axiom_of_choice}
The axiom of choice is unnecessary in this paper, and we avoid using it.
It is well known that the Heine--Borel theorem can be proved without using the axiom of choice;
see, e.g., Lebesgue's proof
explained in~\cite[Chapter~2, Section~5, Theorem~15]{Royden_1988}.
In the proof of the forthcoming
Lemma~\ref{lem:principal},
to define a function $R\colon[a,b]\to(0,+\infty)$ without using the axiom of choice,
we select $R(x)$ as the maximal number of the form $1/q$, where $q\in\bN$,
that has the needed property.
\end{rem}

\begin{lem}
\label{lem:principal}
Let $X$ be a set and
$d_1,d_2$ be distances on $X$ such that 
$d_1\cong d_2$.
Let $a,b\in\bR$, $a<b$,
and $\eta\in C_{d_1}([a,b],X)$.
Then, for every $\eps>0$,
there exist $m$ in $\bN$
and $t_0,t_1,\ldots,t_m$ in $[a,b]$
such that $a=t_0<t_1<\cdots<t_m=b$ and
\begin{equation}
\label{eq:sum_inequality}
\sum_{k=1}^m d_2(\eta(t_{k-1}),\eta(t_k))
\le (1+\eps)
\sum_{k=1}^m d_1(\eta(t_{k-1}),\eta(t_k)).
\end{equation}
\end{lem}

\begin{proof}
For each $t$ in $[a,b]$,
by the definition of $\cong$,
there exists an open neighborhood $W$ of $\eta(t)$ such that
\[
\forall u\in W\qquad
d_2(u,\eta(t)) < (1+\eps) d_1(u,\eta(t)).
\]
Furthermore, by the continuity of $\eta$,
there exists $r>0$ such that if $s\in[a,b]$ and $|s-t|<r$, then
$\eta(s)\in W$.
Moreover, by the Archimedes--Eudoxus principle,
there exists $q$ in $\bN$ such that $1/q<r$.
Thus, for each $t$ in $[a,b]$,
the following set is nonempty:
\[
D(t) \eqdef
\Bigl\{q\in\bN\colon\
\forall s\in[a,b]\cap(t-1/q,t+1/q)\quad
d_2(\eta(s),\eta(t))
<(1+\eps)d_1(\eta(s),\eta(t))\Bigr\}.
\]
By the induction principle,
$D(t)$ has a minimum element.
We define $R\colon[a,b]\to(0,+\infty)$ by
\[
R(t) \eqdef \frac{1}{\min(D(t))}.
\]
The definition of $R(t)$ implies that for each $t$ in $[a,b]$.
\begin{equation}
\label{eq:R_property}
\forall
s\in[a,b]\cap
(t-R(t),t+R(t))
\qquad
d_2(\eta(s),\eta(t))
<(1+\eps)d_1(\eta(s),\eta(t)).
\end{equation}
Applying Lemma~\ref{lem:finite_cover},
we find $m$ in $\bN$
and $t_0,\ldots,t_m$ in $[a,b]$
satisfying~\eqref{eq:difference_t_less_max_R}.

For each $k$ in $\{1,\ldots,m\}$,
due to~\eqref{eq:difference_t_less_max_R},
we have at least one of the cases (a) or (b)
from Remark~\ref{rem:one_point_is_center_for_the_other}.

In the case (a),
we use~\eqref{eq:R_property} with $t_{k-1}$ instead of $t$ and $t_k$ instead of $s$, and we get
\begin{equation}
\label{eq:neighbors_d2_leq_onepluseps_d1}
d_2(\eta(t_{k-1}),\eta(t_k))
< (1+\eps) d_1(\eta(t_{k-1}),\eta(t_k)).
\end{equation}
In the case (b),
we apply~\eqref{eq:R_property} with $t_k$ instead of $t$ and $t_{k-1}$ instead of $s$, and we also get~\eqref{eq:neighbors_d2_leq_onepluseps_d1}.

Finally, summing~\eqref{eq:neighbors_d2_leq_onepluseps_d1} over all $k$ in $\{1,\ldots,m\}$,
we obtain~\eqref{eq:sum_inequality}.
\end{proof}

\begin{thm}
\label{thm:main}
If $X$ is a set and
$d_1,d_2\colon X\to[0,+\infty)$
are distances on $X$ such that
$d_1\cong d_2$,
then $d_1^\ast=d_2^\ast$.
\end{thm}

\begin{proof}
We suppose that $d_1,d_2\colon X\to[0,+\infty)$ such that $d_1\cong d_2$.
Let $x, y\in X$.
Since $d_1$ and $d_2$
are topologically equivalent,
$\Ga_{d_1}(x, y)
= \Ga_{d_2}(x, y)$.
Let $\ga\in \Ga_{d_1}(x, y)$ and $\eps>0$.
Given a partition $P$ in $\cP$,
we enumerate its elements in the ascending order:
$P=\{p_0,p_1,\ldots,p_N\}$,
$0=p_0<p_1<\cdots<p_N=1$.
Then, for each $j$ in $\{1,\ldots,N\}$,
we apply Lemma~\ref{lem:principal}
to the subpath
$\eta_j\eqdef\ga|_{[p_{j-1},p_j]}$.
Thereby, we get $m_j$ in $\bN$
and a list of real numbers
$t_{j,0},t_{j,1},\dots,t_{j,m_j}$ such that
\[
p_{j-1}
= t_{j,0}
< t_{j,1}
< \cdots
< t_{j,m_j}
= p_j
\]
and
\begin{equation}
\label{eq:sum_d2_ga_leq_sum_d1_ga}
\sum_{k=1}^{m_j}
d_2(\ga(t_{j,k-1}),\ga(t_{j,k}))
\le (1+\eps)
\sum_{k=1}^{m_j}
d_1(\ga(t_{j,k-1}),\ga(t_{j,k})).
\end{equation}
Let $Q$ be the partition consisting of all these points $t_{j,k}$,
where 
$1\le j\le N$ and $1\le k\le m_j$.
Then, $P\subseteq Q$ and $Q\in\cP$.
Summing~\eqref{eq:sum_d2_ga_leq_sum_d1_ga}
over $j$ in $\{1,\ldots,N\}$ we obtain
\[
S(\ga,Q,d_2)\le (1+\eps) S(\ga,Q,d_1).
\]
Therefore,
\[
S(\ga,P,d_2)
\le S(\ga, Q, d_2)
\le (1+\eps) S(\ga,Q,d_1)
\le (1+\eps) L(d_1,\ga).
\]
We conclude that
$L(d_2,\ga)\le L(d_1,\ga)$.
Taking the infimum over $\ga$
in $\Ga_{d_1}(x,y)$,
we get
$d_2^\ast(x,y)\le d_1^\ast(x,y)$.
The inequality
$d_1^\ast(x,y)\le d_2^\ast(x,y)$
follows with a similar reasoning,
swapping the roles of $d_1$ and $d_2$ in Lemma~\ref{lem:principal}.
\end{proof}

\begin{cor}
\label{cor:d1_locsim_d2_eq_d2intr}
If $X$ is a set and $d_1,d_2$ are distances on $X$ such that $d_1\cong d_2$ and $d_2^\ast=d_2$,
then $d_1^\ast=d_2$.
\end{cor}

In Examples~\ref{example:pseudologarithmic_distance_on_halfline} and~\ref{example:pseudohyperbolic_disk},
the distances have the property $d^\ast\cong d$.

\begin{problem}
Find a necessary and sufficient condition for $d^\ast\cong d$.
\end{problem}

Notice that $d_1\cong d_2$ is not
a necessary condition for
$d_1^\ast=d_2^\ast$.
We have seen it in Examples~\ref{example:sqrt}
and~\ref{example:quadrant_angular}.
In that examples, the initial distances where not topologically equivalent.

In the following Example~\ref{example:comb},
two distances $d$ and $\rho$
induce the same topology
and the same finite intrinsic distance
$d^\ast=\rho^\ast$,
but $d$ and $\rho$ are not locally similar.

\begin{example}
\label{example:comb}
Let $X$ be the following ``comb'' in $\bR^2$;
see Figure~\ref{fig:comb}:
\[
X \eqdef
([0, 1] \times \{0\})
\cup
\left(\bigcup_{q\in\bN}
\left\{\frac{1}{q}\right\}\times [0,1]\right).
\]
Notice that the ordinate axis $\{0\}\times\bR$ only intersects $X$ at the origin.

Let $d$ be the Euclidean distance on $X$,
inherited from $\bR^2$.
Define $\rho\colon X^2\to[0,+\infty)$,
\[
\rho(a,b)
\eqdef
\begin{cases}
|a_2-b_2|,\quad \text{if}\quad a_1=b_1;
\\
|a_1-b_1|+a_2+b_2,\quad \text{if}\quad a_1\neq b_1;
\end{cases}
\qquad (a=(a_1,a_2), b=(b_1,b_2)\in X).
\]
It is easy to see that $\rho$ is a distance on $X$.
The distances $d$ and $\rho$ are topologically equivalent; this can be proved by comparing the balls $B_d(a,\de_1)$ and $B_\rho(a,\de_2)$ for each $a$ in $X$, with small $\de_1$ and $\de_2$.

Furthermore, $d^\ast=\rho$.
In other words, we claim that
$d^\ast(a,b)=\rho(a,b)$
for every $a$, $b$ in $X$.
For brevity, we illustrate the idea of the proof for the particular case where
$a=(0,0)$ and $b=(1/q, y)$, $q\in\bN$, $0<y\le 1$.
If $\ga\colon[0,1]\to X$ is a continuous path joining $a$ with $b$,
then, by connectivity arguments,
there exists $u$ in $(0,1)$
such that $\ga(u)=(1/q, 0)$.
It is easy to see that
$L(d,\ga|_{[0,u]})\ge 1/q$
and $L(d,\ga|_{[u,1]})\ge y$;
here the domains of the paths are $[0,u]$ and $[u,1]$;
see Remark~\ref{rem:general_intervals_instead_of_unit_interval}.
Therefore, $L(d,\ga) \ge \rho(a,b)$ for every $\ga$ in $\Ga_d(a,b)$,
and it is easy to conclude that $d^\ast(a,b)=\rho(a,b)$.
In particular, $d\le\rho$.

Let us show that $\rho$ is not locally similar to $d$.
Indeed, taking $a=(0,0)$,
$b_q=(1/q,1/q)$,
and $c_q=(1/q,0)$
for $q$ in $\bN$, we get
\[
\frac{\rho(a,b_q)}{d(a,b_q)}
=
\frac{2/q}{\sqrt{2}/q}
=
\sqrt{2},
\qquad
\frac{\rho(a,c_q)}{d(a,c_q)}
=\frac{1/q}{1/q}
=1.
\]
Taking $q$ large enough,
we can get $b_q$ and $c_q$
arbitrarily close to $a$.
Thus, the quotient $\rho(a,x)/d(a,x)$ does not have any limit as $x$ tends to $a$.
More precisely, the superior and inferior limits defined in~\eqref{eq:Lf_and_lf}
are $L_{\id_X}(a)=\sqrt{2}$
and $\ell_{\id_X}(a)=1$.
This means that $\id_X$ has distortion $\sqrt{2}$ at the point $a$.
\hfill\qedsymbol
\end{example}

\begin{figure}[htp]
\centering
\begin{tikzpicture}
\filldraw [lightgreen]
  (-0.1, -0.1) rectangle (3, 0.1);
\filldraw [lightgreen]
  (3-0.1, -0.1) rectangle (3+0.1, 3+0.1);
\draw [thick] (0, 0) -- (6, 0);
\node at (0, -0.1) [below]
  {$\scriptstyle (0,0)$};
\draw [thick] (6, 0) -- (6, 6);
\node at (6, -0.1) [below]
  {$\scriptstyle (1, 0)$};
\draw [thick] (3, 0) -- (3, 6);
\node at (3, -0.1) [below]
  {$\scriptstyle \left(\frac{1}{2}, 0\right)$};
\node at (3.1, 3) [right]
  {$\scriptstyle \left(\frac{1}{2},\frac{1}{2}\right)$};
\draw [thick] (2, 0) -- (2, 6);
\node at (2, -0.1) [below]
  {$\scriptstyle \left(\frac{1}{3}, 0\right)$};
\draw [thick] (1.5, 0) -- (1.5, 6);
\draw [thick] (1.2, 0) -- (1.2, 6);
\draw [thick] (1, 0) -- (1, 6);
\node at (0.5, 3) {\Large\dots};
\end{tikzpicture}
\caption{Part of the comb from
Example~\ref{example:comb}.
The light-green thick lines illustrate the shortest path from $(0,0)$ to $\left(\frac{1}{2},\frac{1}{2}\right)$.
\label{fig:comb}}
\end{figure}

\medskip
\section{Examples related to kernels}
\label{sec:examples_RK}

Let $X$ be a set.
A function $K\colon X^2\to\bC$ is called a \emph{kernel} on $X$ if for every $m$ in $\bN$, every $x_1,\ldots,x_m$ in $X$ and every $\la_1,\ldots,\la_m$ in $\bC$,
\begin{equation}
\label{eq:kernel_definition}
\sum_{r,s=1}^m \overline{\la_r}\,\la_s\,K(x_r,x_s)\ge0.
\end{equation}
By the Moore--Aronszajn theorem
(see~\cite[Part~1, Section~4]{aronszajn_theory}
or \cite[Theorem~2.14]{Paulsen_Raghupathi_2016_introduction}),
every kernel corresponds to a uniquely determined reproducing kernel Hilbert space (RKHS). Conversely, every RKHS has a uniquely defined kernel.
So, speaking about kernels is equivalent to speaking about RKHS.

A function $K\colon X^2\to\bC$ is called a \emph{strictly positive kernel} on $X$ if for every $m$ in $\bN$, every list of pairwise different points $x_1,\ldots,x_m$ in $X$
and every
$\la_1,\ldots,\la_m$ in $\bC$,
not all zero,
\begin{equation}
\label{eq:strict_kernel_definition}
\sum_{r,s=1}^m \overline{\la_r}\,\la_s\,K(x_r,x_s)>0.
\end{equation}
See various characterizations of strictly positive kernels in~\cite[Section~3.2]{Paulsen_Raghupathi_2016_introduction}.

Agler and McCarthy~\cite[Lemma~9.9]{Agler_McCarthy_2002_Pick_interpolation}
proved that the following function $d_K$
is a distance on $X$.
See also
Arcozzi,
Rochberg,
Sawyer,
and Wick~\cite{Arcozzi2011}.

\begin{defn}[the distance associated with a kernel]
Let $X$ be a set and $K$ be a strictly positive kernel on $X$.
Define
$d_K\colon X^2\to[0,+\infty)$ by
\begin{equation}
d_K(x, y)
\eqdef
\sqrt{1-\frac{|K(x,y)|^2}{K(x,x)K(y,y)}}.
\end{equation}
\end{defn}

In this section, we apply the previous tools,
including Theorem~\ref{thm:main} and Corollary~\ref{cor:d1_locsim_d2_eq_d2intr},
to compute $d_K^\ast$ for some examples.

We start with two classical examples (Examples~\ref{example:Bergman_halfplane} and~\ref{example:Bergman_disk})
that are usually studied with the help of the \emph{Bergman metrics}
(for these two examples, the Bergman metrics are constant multiples of the Poincar\'{e} metrics),
see~\cite{Bergman1970,Zhu2005}.
We show that these examples can be treated with more elementary tools.

\begin{example}[Bergman space on the upper halfplane]
\label{example:Bergman_halfplane}
Consider the Bergman space of the analytic square-integrable functions on the upper halfplane with the Lebesgue measure.
The reproducing kernel of this space is
\[
K(w,z)
=\frac{1}{\pi(w-\overline{z})^2}.
\]
The associated distance is
\[
d_K(w,z)
=
\left|\frac{w-z}{w-\conj{z}} \right|\sqrt{2-\left|\frac{w-z}{w-\conj{z}} \right|^2}.
\]
So, $d_K$ can be written as the composition $f\circ \rho_\bH$,
where $\rho_\bH$ is as defined in~\eqref{eq:def_pseudohyperbolic_distance_on_halfplane}, and
\[
f(t)\eqdef t\sqrt{2-t^2}.
\]
It is easy to see that $f$ satisfies conditions of Corollary~\ref{cor:f_concave} with $Q = \sqrt{2}$. 
Therefore, $d_K\cong \sqrt{2}\,\rho_\bH$.
By Theorem~\ref{thm:main},
$d_K^\ast = \sqrt{2}\,\rho_\bH^\ast$.
Recall that $\rho_\bH^\ast$
is the hyperbolic distance on $\bH$,
see Example~\ref{example:pseudohyperbolic_halfplane}.
\hfill\qedsymbol
\end{example}

\begin{example}[Bergman space on the unit disk]
\label{example:Bergman_disk}
Consider the Bergman space of the analytic square-integrable functions on $\bD$ with the Lebesgue measure;
see, e.g., \cite{Bergman1970,Zhu2005}.
The reproducing kernel of this space is
\[
K(w,z) = \frac{1}{\pi(1-w\overline{z})^2}.
\]
The associated distance is
\[
d_K(z, w) 
= 
\left|\frac{z-w}{1-\overline{z}w} \right| \sqrt{2-\left| \frac{z-w}{1-\overline{z}w}\right|^2}.
\]
Note that $d_K = f\circ \rho_{\bD}$,
where $\rho_{\bD}$ is as in~\eqref{eq:def_pseudohyperbolic_distance_on_disk} and
\[
f(t) \eqdef t\sqrt{2-t^2}.
\]
Reasoning as in Example~\ref{example:Bergman_halfplane}, by Corollary~\ref{cor:f_concave},
$d_K \cong \sqrt{2}\,\rho_\bD$.
By Theorem~\ref{thm:main},
$d_K^\ast = \sqrt{2}\,\rho^\ast_\bD$.
\hfill\qedsymbol
\end{example}

\begin{example}[Hardy space and Szeg\H{o} kernel on the disk]
This example is studied in~\cite{Arcozzi2011}.
The reproducing kernel of the Hardy space, known as \emph{Szeg\H{o} kernel}, is
\[
K(w, z)
= \frac{1}{1-w\overline{z}}.
\]
The associated distance is the pseudohyperbolic distance: $d_K = \rho_{\bD}$. Therefore, $d_K^\ast = \rho_{\bD}^\ast$.
\hfill\qedsymbol
\end{example}

\begin{example}[the Gaussian kernel]
Let $\sigma>0$.
We consider the following kernel on $\bR^n$:
\[
K(x,y)
=\exp(-\sigma^2 \|x-y\|^2).
\]
The corresponding RKHS was described
by Saitoh~\cite{Saitoh1983}
and by Steinwart, Hush, and Scovel~\cite{SteinwartHushScovel2006}.
In this example,
\[
d_K(x,y)
=\sqrt{1-\exp(-2\sigma^2\|x-y\|^2)}.
\]
So, $d_K$ can be written as the composition
$f\circ d_{\bR^n}$,
where
\begin{equation}
\label{eq:sqrt(1-exp(-t2))}
f(t)\eqdef
\sqrt{1-\enumber^{-2\sigma^2 t^2}}.
\end{equation}
It is easy to see that $f$ satisfies the conditions from
Corollary~\ref{cor:f_concave}
with $Q = \sqrt{2}\sigma$.
Therefore,
$d_K\cong \sqrt{2}\sigma d_{\bR^n}$.
By Theorem~\ref{thm:main}
and Example~\ref{example:Euclidean_distance},
$d_K^\ast
=\sqrt{2}\sigma d_{\bR^n}^\ast
=\sqrt{2}\sigma d_{\bR^n}$.
\hfill\qedsymbol
\end{example}

\begin{example}[Bargmann--Segal--Fock space on the complex plane]
\label{example:Fock}
For the definition and properties of this space, see~\cite{Zhu2012}.
Its reproducing kernel is
\[
K(w,z)
= \enumber^{w\conj{z}}.
\]
In this example, $d_K\colon\bC^2\to[0,+\infty)$ is defined by
\[
d_K(w,z)
= \sqrt{1-\enumber^{-|z-w|^2}}\qquad(w,z\in\bC).
\]
We notice that $d_K = f\circ d_\bC$, where $d_\bC$ is the Euclidean distance on $\bC$ and
\[
f(t)\eqdef\sqrt{1-\enumber^{-t^2}}.
\]
By Corollary~\ref{cor:f_concave},
$d_K\cong d_\bC$.
By Theorem~\ref{thm:main} and Example~\ref{example:Euclidean_distance},
$d_K^\ast = d_\bC^\ast = d_\bC$.
\hfill\qedsymbol
\end{example}

\begin{example}[polyanalytic Bargmann--Segal--Fock space on the complex plane]
\label{example:poly_Fock}
This is a generalization of Example~\ref{example:Fock}.
Let $m\ge 1$.
Consider the space of the $m$-analytic functions on $\bC$, square integrable with respect to the Gaussian measure.
The corresponding reproducing kernel was computed by Askour, Intissar, and Mouayn~\cite{AskourIntissarMouayn1997}:
\begin{equation}
\label{eq:kernel_poly_Fock}
K(w,z)
= \enumber^{w\,\overline{z}}
L^1_{m-1}(|w-z|^2).
\end{equation}
See also Youssfi~\cite{Youssfi2021} for another proof of~\eqref{eq:kernel_poly_Fock},
in the multidimensional case.
Here $L^1_{m-1}$ denotes the associated Laguerre(-Sonine) polynomial of degree $m-1$,
with parameter $1$:
\begin{equation}
\label{eq:associated_Laguerre}
L^1_{k}(x)
= 
\sum_{j=0}^{k}
(-1)^j \binom{k+1}{k-j}
\frac{x^j}{j!}.
\end{equation}
The associated distance is
\begin{equation}
\label{eq:distance_poly_Fock}
d_K(w,z)
=
\sqrt{1-\frac{1}{m^2}
\enumber^{-|w-z|^2}
L^1_{m-1}(|w-z|^2)^2}
\qquad(w,z\in\bC).
\end{equation}
We see that
$d_K=f\circ d_{\bC}$, where
\[
f(t) \eqdef
\sqrt{1-\frac{1}{m^2}\enumber^{-t^2}L^1_{m-1}(t^2)^2}.
\]
It follows from~\eqref{eq:associated_Laguerre}
that
$L_{m-1}^1(t^2)^2=m^2-m^2(m-1)t^2+O(t^4)$.
Therefore, an elementary analysis shows that $f$ satisfies conditions of Proposition~\ref{prop:d_2 = f circ d_1}, with $Q=\sqrt{m}$.
Therefore,
$d_K\cong \sqrt{m}\,d_{\bC}$.
By Theorem~\ref{thm:main} and Example~\ref{example:Euclidean_distance},
$d_K^\ast
= \sqrt{m}\,d_\bC^\ast
= \sqrt{m}\,d_\bC$.
\hfill\qedsymbol
\end{example}

\begin{example}[the space of band-limited functions]
Given $f$ in $L^2(\bR)$, we denote by $\widehat{f}$ the Fourier--Plancherel transform of $f$,
with the factor $\exp(-2\pi\iu xt)$ under the integral.
Let $A>0$.
The space of band-limited functions with halfwidth $A$,
also known as the Paley--Wiener space of halfwidth $A$,
is defined by
\[
\mathcal{PW}_A
\eqdef
\bigl\{\widehat{g}\colon \ g\in L^2(\bR),\ 
\mu(\{x\in\bR\setminus[-A,A]\colon\ g(x)\ne0\})=0\bigr\}.
\]
This space is considered with the inner product
$\langle \widehat{g_1}, \widehat{g_2}\rangle_{\mathcal{PW}_A}
\eqdef
\langle g_1, g_2\rangle_{L^2(\bR)}$
which coincides with
$\langle \widehat{g_1}, \widehat{g_2}\rangle_{L^2(\bR)}$ by the Plancherel theorem.
We denote by $\sinc$ the normalized cardinal sine function defined on $\bR$ by
\[
\sinc(x)\eqdef
\begin{cases}
1,
&
x=0;
\\
\frac{\sin(\pi x)}{\pi x},
&
x\in\bR\setminus\{0\}.
\end{cases}
\]
Using the formula
\begin{equation}
(\widehat{1_{[-A,A]}})(t)
= 2A\sinc(2At)
\qquad(t\in\bR)
\end{equation}
it is easy to see (e.g., \cite[Section 4.3, Example 4.2]{manton_primer},
\cite[Subsection~1.3.2]{Paulsen_Raghupathi_2016_introduction})
that $\mathcal{PW}_A$ is a RKHS on $\bR$, and its reproducing kernel is
\begin{equation}
\label{PW_RK}
K(s,t)
= 2A\sinc(2A(s-t)).
\end{equation}
The associated distance is
\begin{equation}
\label{eq:dK_PW}
d_K(s, t)
= \sqrt{1-\sinc^2(2A(t-s))}.
\end{equation}
Notice that
$d_K=f\circ\left(\frac{2A\pi}{\sqrt{3}}\,d_\bR\right)$, where
\[
f(t) \eqdef
\sqrt{1-\sinc^2\frac{\sqrt{3}\,t}{\pi}}.
\]
A simple calculation with power series yields
$f'(0)=1$.
Using the identity
$\sin^2(x) = \frac{1-\cos(2x)}{2}$, we get
\[
\sinc^2(x)
= \frac{1-\cos(2\pi x)}{2(\pi x)^2}
\leq 
1-\frac{(\pi x)^2}{3}+ \frac{2(\pi x)^4}{45}.
\]
If $t\in [0, 1]$, then
\[
f(t)
=\sqrt{1-\sinc^2\frac{\sqrt{3}\,t}{\pi}}
\ge
\sqrt{t^2 - \frac{2t^4}{5}}
=
\sqrt{t^2\left(1-\frac{2t^2}{5}\right)}
\ge 
\frac{t}{2}.
\]
Recall that
$|\sinc(x)|\le\frac{1}{\pi x}$.
If $t>1$, then
\[
f(t)
=
\sqrt{1-\sinc^2\frac{\sqrt{3}\,t}{\pi}}
\ge
\sqrt{1-\frac{1}{3t^2}}
>\frac{1}{2}.
\]
Joining the two cases, we get
\[
f(t)\ge\frac{1}{2}\min\{t,1\}
\qquad(t>0).
\]
This means that $f$ satisfies the hypothesis of Proposition~\ref{prop:d_2 = f circ d_1} with $Q=1$.
So, 
$d_K\cong\frac{2A\pi}{\sqrt{3}}d_\bR$,
and
$d_K^\ast = \frac{2A\pi}{\sqrt{3}}d_\bR^\ast
= \frac{2A\pi}{\sqrt{3}}d_\bR$.
\hfill\qedsymbol
\end{example}

\begin{example}[a Sobolev space]
\label{example:Sobolev}
We denote by $\AC([0, 1])$ the set of all absolutely continuous complex-valued functions on $[0, 1]$.
Consider the vector space
\[
\cH \eqdef
\bigl\{f\in \AC([0, 1]) \colon \ f(0) = 0,\ f(1)=0,\ f'\in L^2([0,1])\bigr\},
\]
with the inner product
$\langle f, g\rangle_{\cH}
\eqdef \langle f',g'\rangle_{L^2([0,1])}$.
It is known~\cite[Subsection 1.3.1]{Paulsen_Raghupathi_2016_introduction} 
that $\cH$ is an RKHS, and its reproducing kernel is
\begin{equation}
K(x,y)
=
\begin{cases}
(1-x)y, & y\le x; \\
(1-y)x, & x<y.
\end{cases}
\end{equation}
The distance associated to $K$ is
\begin{equation}\label{eqn:d_K_Sobolev_space}
d_K(x, y)
= 
\begin{cases}
0, & x=y;
\\
\sqrt{\frac{y-x}{(1-y)x}},
& y>x;
\\
\sqrt{\frac{x-y}{(1-x)y}},
& x>y.
\end{cases}
\end{equation}
Since $x\le 1$ and $y\le 1$,
\[
d_K(x,y)\ge \sqrt{|x-y|}.
\]
Reasoning similarly to Example~\ref{example:sqrt},
we see that
\begin{equation}
\label{eq:d_K_intrinsic_Sobolev}
d_K^\ast(x, y)
= 
\begin{cases}
0, & x=y;
\\
+\infty, & x\neq y.
\end{cases}
\end{equation}
\hfill\qedsymbol
\end{example}

\begin{example}[the min-kernel on the unit interval]
Consider the vector space
\[
\cH \eqdef
\bigl\{f\in \AC([0, 1]) \colon \ f(0) = 0, \ f'\in L^2([0,1])\bigr\}
\]
with the inner product
$\langle f, g\rangle_{\cH}
\eqdef \langle f',g'\rangle_{L^2([0,1])}$.
It is known~\cite[Section 4.3, Example 4.3]{manton_primer} 
that $\cH$ is an RKHS, and its reproducing kernel is
\begin{equation}
K(x,y)
=
\min\{x, y\}.
\end{equation}
This kernel corresponds to the covariance for the Wiener process.
The induced distance is
\begin{equation}
\label{eqn:d_K_min}
d_K(x, y)
= 
\begin{cases}
0, & x=y;
\\
\sqrt{\frac{y-x}{y}},
& y>x;
\\
\sqrt{\frac{x-y}{x}},
& x>y.
\end{cases}
\end{equation}
Since $x\le 1$ and $y\le 1$,
\[
d_K(x,y)\ge \sqrt{|x-y|}.
\]
Similarly to Example~\ref{example:Sobolev},
$d_K^\ast$ is given by~\eqref{eq:d_K_intrinsic_Sobolev}.
\hfill\qedsymbol
\end{example}

\medskip
\section*{Acknowledgements}

We are grateful to Prof. John Harvey (\myurl{https://orcid.org/0000-0001-9211-0060})
for the bibliographic reference~\cite{Lytchak2005} which we use in Section~\ref{sec:strongly_similar_distances}, Definition~\ref{def:infin_similar}.
We are also grateful to Prof. David Jos\'{e} Fern\'{a}ndez-Bret\'{o}n
(\myurl{https://orcid.org/0000-0001-8084-055X})
for explaining to us the ideas from Remark~\ref{rem:axiom_of_choice}.
We are grateful to Prof. Miguel Angel Rodriguez Rodriguez
(\myurl{https://orcid.org/0000-0002-5124-8271})
for revising a draft of this paper and correcting some minor errors.
We are grateful to Omer Cantor
(\myurl{https://orcid.org/0009-0007-2969-7170})
for Example~\ref{example:OmerCantor}.

\medskip
\section*{Funding}

The second author has been supported by Proyecto SECIHTI (Mexico) ``Ciencia de Frontera''
FORDECYT-PRONACES/61517/2020
and by IPN-SIP projects (Instituto Polit\'{e}cnico Nacional, Mexico).
The third author has been supported by SECIHTI (Mexico) scholarship.

\bigskip
\subsection*{Information about the authors}

Erick Lee-Guzm\'{a}n,
\myurl{https://orcid.org/0009-0001-1731-8803},
\\
email: ericklee81@gmail.com

\medskip\noindent
Egor A. Maximenko,
\myurl{https://orcid.org/0000-0002-1497-4338},
\\
email: egormaximenko@gmail.com, emaximenko@ipn.mx

\medskip\noindent
Enrique Abdeel Mu\~{n}oz-de-la-Colina,
\myurl{https://orcid.org/0009-0002-2397-6607},
\\
email: abdeeldlc@gmail.com

\medskip\noindent
Marco Iv\'{a}n Ruiz-Carmona,
\myurl{https://orcid.org/0009-0008-8639-6291},
\\
email: marco.ivan.ruiz.0821@gmail.com

\medskip\noindent
\textbf{Affiliation:}
Instituto Polit\'{e}cnico Nacional,
Escuela Superior de F\'{i}sica y Matem\'{a}ticas,
\\
Ciudad de M\'{e}xico,
Postal Code 07738,
Mexico.

\end{document}